\newcommand{\dataversione}{May 15, 2019}
\documentclass[11pt,reqno,a4paper,final]{amsart}
\usepackage{fancyhdr}
\usepackage{amssymb}
\usepackage{mathrsfs}
\usepackage{paralist}
\usepackage{algorithm}
\usepackage{graphicx}
\usepackage[T1]{fontenc}
\usepackage[colorlinks=true]{hyperref}\hypersetup{urlcolor=blue, citecolor=red}
\usepackage{enumitem}\setlist[itemize]{leftmargin=28 pt}

\makeatletter
\def\@secnumfont{\bfseries}
\def\section{\@startsection%
{section}
{1}
\z@{1.5\linespacing\@plus.2\linespacing}
  {.5\linespacing}
  {\large\normalfont\bfseries}}
\def\subsection{\@startsection%
{subsection}
{2}
\z@{.5\linespacing\@plus.2\linespacing}
  {-.5em}
  {\normalfont\bfseries}}
\makeatother

\numberwithin{equation}{section}

\newtheorem{theorem}[subsection]{Theorem}
\newtheorem{lemma}[subsection]{Lemma}
\newtheorem{proposition}[subsection]{Proposition}

\theoremstyle{definition}
\newtheorem{remark}[subsection]{Remark}
\newtheorem{example}[subsection]{Example}
\newtheorem{problem}[subsection]{Problem}

\DeclareMathOperator{\dive}{div}
\DeclareMathOperator{\dist}{dist}
\DeclareMathOperator{\trace}{\mbox{\Large$\!\llcorner$}}

\newcommand{\eps}{\varepsilon}
\newcommand{\bd}{\partial}
\newcommand{\Lip}{\mathrm{Lip}}
\newcommand{\areas}{\mathrm{areas}}
\newcommand{\weights}{\mathrm{weights}}
\newcommand{\lengths}{\mathrm{lengths}}
\newcommand{\R}{\mathbb{R}}
\newcommand{\HH}{\mathscr{H}}
\newcommand{\M}{\mathscr{M}}
\newcommand{\ML}{\M_L}
\newcommand{\MLa}{\M_L^a}
\newcommand{\A}{\mathscr{A}}
\newcommand{\T}{\mathcal{T}}
\newcommand{\AL}{\A_L}
\newcommand{\EE}{\mathcal{E}}

\newcommand{\Omegabar}{{\smash{\overline\Omega}}}
\newcommand{\Golab}{Go{\l}\k{a}b}
\newcommand{\nablatau}{\nabla_{\!\tau}}

\begin{document}

\thispagestyle{empty}

~\vskip -1.1 cm

	%
	%
{\footnotesize\noindent 
[version: \dataversione]
\hfill 
to appear on \emph{Netw. Heterog. Media}\par
}

\vspace{1.7 cm}

	%
	%
{\centering\Large\bf
Optimal reinforcing networks for elastic membranes
\\
}

\vspace{.7 cm}

	%
	%
{\centering\sc 
Giovanni Alberti, Giuseppe Buttazzo,
Serena Guarino Lo Bianco, 
\\
\smallskip
and \'Edouard Oudet 
\\
}

\vspace{.6 cm}

	%
	%
{\rightskip 1 cm
\leftskip 1 cm
\parindent 0 pt
\footnotesize
{\sc Abstract.}
In this paper we study the optimal reinforcement of an elastic membrane, 
fixed at its boundary, by means of a network (connected one-dimensional structure), 
that has to be found in a suitable admissible class. 
We show the existence of an optimal network, and observe that such network 
carries a multiplicity that in principle can be strictly larger than one. 
Some numerical simulations are shown to confirm this issue and to illustrate 
the complexity of the optimal network when the total length becomes large.
\par
}

\section{Introduction}
\label{sintro}
In the present paper we consider the vertical displacement of an elastic membrane 
under the action of an exterior load $f$ and fixed at its boundary; 
this amounts to solve the variational problem
\begin{equation}
\label{pb1}
\min\left\{ 
   \frac{1}{2}\int_\Omega|\nabla u|^2\,dx -\int_\Omega fu\,dx ~\colon u\in H^1_0(\Omega) 
\right\}
\,,
\end{equation}
or equivalently the elliptic PDE
\[
\text{$-\Delta u=f$ in $\Omega$,\qquad $u\in H^1_0(\Omega)$.}
\]
Here $\Omega$ is a bounded Lipschitz domain of $\R^2$, $f\in L^2(\Omega)$, 
and $H^1_0(\Omega)$ is the usual Sobolev space of functions with zero 
trace on the boundary $\bd\Omega$.

Our goal is to rigidify the membrane by adding a one-dimensional reinforcement 
in the most efficient way; the reinforcement is described by a one-dimensional
set $S\subset\Omega$ which varies in a suitable class of admissible choices. 
The effect of $S$ on the membrane is described by the energy
\begin{equation}
\label{energy}
\EE_f(S) := 
\inf\bigg\{
   \frac{1}{2}\int_\Omega|\nabla u|^2\,dx 
   + \frac{m}{2}\int_S|\nabla u|^2\,d\HH^1 
   - \int_\Omega fu\,dx
   ~\colon u\in C^\infty_c(\Omega)
\bigg\}
\end{equation}
that has to be maximized in the class of admissible choices for $S$. 

Here $m>0$ is a fixed parameter that represents the stiffness coefficient 
of the one-dimensional reinforcement, $\HH^1$ denotes the 
1-dimensional Hausdorff measure (that is, the length measure), 
while $C^\infty_c(\Omega)$ denotes the class of smooth functions with 
compact support in $\Omega$. 

The optimization problem we deal with consists in finding the 
``best'' reinforcement $S$ among all networks with total length 
bounded by a prescribed $L$, that is, all $S$ in the class
\[
\AL 
:=\Big\{ 
  \text{$S$ closed connected subset of $\Omega$ with $\HH^1(S)\le L$} 
\Big\}
\, , 
\]
We then consider the maximization of the energy functional $\EE_f(S)$ 
in \eqref{energy} over this class, that is
\begin{equation}
\label{optpb}
\max\Big\{ \EE_f(S) \colon S \in\AL \Big\}
\, .
\end{equation}

\subsection{Gradient versus tangential gradient}
%
From the modeling point of view, it is natural to ask whether the 
gradient $\nabla u$ that appears in the line integral 
\[
\frac{m}{2} \int_S|\nabla u|^2\,d\HH^1
\]
in \eqref{energy} should be replaced by the tangential gradient
$\nablatau u$. 
It turns out that the question is irrelevant, at least if we strictly
follow a variational approach, because the value of $\EE_f(S)$ 
is not affected by this change (Theorem~\ref{equivalence}).

Indeed, if $S$ is a compact curve of class $C^1$ contained in 
$\Omega$, it is well-known (see for instance \cite{bobuse}) that 
the relaxation of the integral
\[
F(u):=\int_S |\nabla u|^2\, d\HH^1
\, , \quad
u\in C^\infty_c(\Omega)
\, , 
\]
is given by
\[
F^*(u):= \int_S |\nablatau u|^2\, d\HH^1
\, , \quad
u\in H^1(S)
\, .
\]

This relaxation result holds also when $S$ is a compact connected
set with finite length, provided that $H^1(S)$ and $\nablatau$ 
are properly defined (this statement
is implicitly contained in Proposition~\ref{emumin}).
However, we warn the reader that this relaxation result does 
not holds if $S$ is an arbitrary compact subset of a curve of class 
$C^1$ with positive length; in particular, if $S$ is totally 
disconnected then the relaxation of $F$ is equal to $0$ 
for every $u$.

\subsection{Concentrated loads}
\label{conload}
Besides the case of \emph{distributed loads}, which consists 
in assuming that $f$ belongs to some Lebesgue class $L^p(\Omega)$, 
we may also consider the case 
of \emph{concentrated loads}, in which $f$ may have a more singular behavior. 
More precisely, we may assume that $f$ is a signed measure on $\Omega$. 
(In this case the linear term $\int_\Omega fu\,dx$ in \eqref{energy} should be 
written as $\int_\Omega u\,df$.)
%

We recall that a measure $f$ does not necessarily belong 
to the dual of the Sobolev space $H^1_0(\Omega)$, 
and therefore $\EE_f(S)=-\infty$ for some choices of $S$. 
Clearly, if $\EE_f(S)$ is finite for at least one $S$
problem~\eqref{optpb} still makes sense, and we may discard
all $S$ such that $\EE_f(S)=-\infty$.
However, it may happen that $\EE_f(S)=-\infty$ for every $S$ in $\AL$
and in that case problem~\eqref{optpb} does not makes sense
(see Example~\ref{example}).

\subsection{An optimization problem in a model for traffic congestion}
Another optimization problem requiring a similar analytical approach 
arises in a model for the reduction of traffic congestion in a given geographic area. 
Here the minimum problem is
\begin{equation}
\label{traffic}
\min\left\{
  \int_\Omega H(\sigma)\,dx 
  ~\colon \text{$-\dive \sigma=f$ in $\Omega$, $\sigma\cdot n=0$ on $\bd\Omega$}
\right\}
\, , 
\end{equation}
where $f=f^+-f^-$ and in the region $\Omega$ the function $f^+$ represents
the density of residents while $f^-$ is the density of working places.
The vector $\sigma$ is the traffic flux and the function $H$ describes
the transportation cost; the case $H(s):=|s|$ gives the classical Monge's
problem, while we talk of \emph{congested transport} if the function $H$ 
is super-linear at infinity, that is,
\[
\lim_{|s|\to\infty}\frac{H(s)}{|s|}=+\infty
\, .
\]
We refer to \cite{be52, bracarsan, bucagu, wa52} and to the references 
therein for a detailed description of this model. 
In the case $H(s):=|s|^2/2$, the minimization problem~\eqref{traffic} 
reduces, via a duality argument, to a problem of the form~\eqref{pb1}.

The optimization problem arises when a new road, or network of roads, 
$S$ has to be built to reduce the congestion; the total length $L$ 
is prescribed and on the new road the congestion function is strictly 
lower than $|s|^2/2$, for example $\alpha|s|^2/2$ with $\alpha<1$. 
The problem then consists in finding the optimal one-dimensional 
set $S$, and we end up, via a duality argument, with a problem similar
to \eqref{optpb}, with $m:=1/\alpha$.

\subsection{Relaxed formulation of the optimization problem}
\label{relaxed}
The optimization problem \eqref{optpb} is solved, 
\emph{in a suitable relaxed form}, in Section~\ref{sform}, 
to which we refer for precise statements and definitions.

We explain first the need for a relaxed formulation.
Consider a maximizing sequence $(S_n)$ for problem
\eqref{optpb}: since these sets are closed, connected, and satisfy
$\HH^1(S_n)\le L$, they converge, up to subsequence 
and in Hausdorff distance, to some connected
compact set $S_\infty$ with $\HH^1(S_\infty)\le L$
contained in the closure $\Omegabar$. 
The problem is that the functional $\EE_f(S)$ is not
upper semicontinuous in $S$ with respect to Hausdorff convergence, 
and therefore $S_\infty$ may be not a solution of problem~\eqref{optpb}.

However it turns out that $\EE_f(S)$ is upper semicontinuous
if we identify the sets $S$ with the measures $\HH^1\trace S$, namely
the restrictions of the Hausdorff measure $\HH^1$ to $S$, 
and consider the weak* convergence of measures instead of the 
Hausdorff convergence of sets.
More precisely, we extend the energy functional 
\eqref{energy} to general positive 
measures $\mu$ on $\Omegabar$ by setting
\[
\EE_f(\mu):=
\inf\bigg\{
  \frac{1}{2}\int_\Omega|\nabla u|^2\,dx 
  + \frac{m}{2}\int_\Omega|\nabla u|^2\,d\mu
  - \int_\Omega u\,df
  \colon u\in C^\infty_c(\Omega)
\bigg\}
\]
(here we assume that the load $f$ is a signed measure).
Notice that this new functional extends the previous
one in the sense that $\EE_f(\mu)=\EE_f(S)$ when 
$\mu=\HH^1\trace S$, and it is upper semicontinuous 
with respect to the weak* convergence of~$\mu$
(Proposition~\ref{usc}).
 
The problem now is that weak* limits of measures of the form 
$\HH^1\trace S$ are not necessarily measures of the same form,
and in particular the limit $\mu_\infty$ of the measures 
$\HH^1\trace S_n$ is a measure supported
on the set $S_\infty$, but may be not the measure
$\HH^1\trace S_\infty$; if it is, then $S_\infty$ is a solution 
of the optimization problem~\eqref{optpb}, but otherwise it is not.

These considerations lead to the following relaxed version of 
problem~\eqref{optpb}:
\begin{equation}
\label{maxbl}
\max\Big\{\EE_f(\mu) \colon \mu\in\ML\Big\}
\, , 
\end{equation}
where $\ML$ is the class of all weak* limits of measures of the 
form $\HH^1\trace S$ with~$S$ admissible network, that is,~$S\in\AL$.

The class $\M_L$ is completely 
described in Proposition~\ref{closure},
and the existence of a solution
of the relaxed optimization problem~\eqref{maxbl}
is proved in Theorem~\ref{exthml}. 
In Theorem~\ref{exthsl} we show
that there is always a solution of the form 
$\mu=\theta\,\HH^1\trace S$
where $S$ is a compact, connected set
with finite length contained in $\Omegabar$, 
and $\theta$ is a real-valued multiplicity function 
which satisfies $\theta(x)\ge 1$ for $\HH^1$-a.e.~$x\in S$.

If $\theta(x)= 1$ for $\HH^1$-a.e.~$x\in S$
then $S$ is a solution of the original 
optimization problem~\eqref{optpb}. If not, then 
problem~\eqref{optpb} may have no solution.

In Section~\ref{snum} we present some numerical simulations 
which show unexpected behaviors of the optimal measures $\mu$;
in particular we have evidence that in some situations
(and perhaps most situations) the multiplicity $\theta$ 
may be strictly larger 
than $1$ in a subset of positive length of $S$.

\subsection{Final remarks}
%
(i)~We do not know if problem~\eqref{maxbl}
is ``the'' relaxation of problem~\eqref{optpb}, 
and in particular we cannot exclude that some 
kind of Lavrentiev phenomenon occurs (for more details 
see Problem~\ref{relax} and Remark~\ref{lavrentiev}).

\smallskip
(ii)~The connectedness assumption on $S$ is crucial: 
indeed, removing this constraint allows a sequence of 
maximizing sets $S_n$ to spread all over $\Omega$ and 
leads to a relaxed problem of the form
\[
\max\Big\{\EE_f(\mu) \colon \mu\in\M^+(\Omegabar)\, , \ \mu(\Omegabar)\le L \Big\}
\,,
\]
where $\M^+(\Omegabar)$ is the class of all positive measures on $\Omegabar$. 
This optimization problem has been studied in \cite{buva05} and in \cite{buouve15}, 
where it is shown that the optimal measure $\mu$ actually belongs to $L^p(\Omega)$ 
and the exponent $p$ depends on the summability of the right-hand side $f$. 
Similar problems, in the extreme case when in the reinforcing region a Dirichlet 
condition is imposed, have been considered in \cite{busa07, busava06}.

\smallskip
(iii)~In the definitions of $\EE_f(S)$ and $\EE_f(\mu)$
we required that $u$ belongs to $C^\infty_c(\Omega)$ 
to ensure that all integrals makes sense.
Clearly it would be equivalent to consider 
continuous functions in $H^1_0(\Omega)$ that are of class $C^1$
in a neighborhood of $S$ or of the support of $\mu$.
One can go further, and take $u$ in a suitably defined Sobolev space, 
so that the infimum in the definition of $\EE_f(\mu)$
is a minimum (Proposition~\ref{emumin}).

\smallskip
(iv)~In our model the stiffener $S$ is a one-dimensional set 
and its contribution to the total energy is described by the line integral
\[
\frac{m}{2}\int_S|\nablatau u|^2\,d\HH^1
\, .
\]
This choice is consistent with the fact that the integral above is the 
variational limit as $\eps\to0$ of the integrals 
\[
\frac{m}{2\eps}\int_{S_\eps}|\nabla u|^2\,dx
\, , 
\]
where $S_\eps$ is the thin strip 
$S_\eps:=\big\{x\in\R^2 \colon \dist(x,S)<\eps/2\big\}$.
In other words, the one-dimensional stiffener $S$ can be seen as the limit 
structure of two-dimensional thin strips of thickness $\eps$ and 
elastic constants $m/\eps$ 
(see for instance \cite{sapa80} and \cite{acbupe88}).

\subsection*{Structure of the paper}
In Section~\ref{sform} we give a precise formulation of
the relaxed optimization problem, and state the main 
existence results (Theorems~\ref{exthml} and~\ref{exthsl}).
In Section~\ref{sproofs} we prove the results stated 
in Section~\ref{sform}.
In Section~\ref{snece} we give some additional
properties that solution of the relaxed optimization 
problem must satisfy.
Section~\ref{snum} is devoted to the numerical approximation 
of the relaxed problem. 
Section~\ref{sopen} contains additional remarks
and open problems.

\subsection*{Acknowledgements}
This work is part of the PRIN projects 2017TEXA3H and 2017BTM7SN, funded by 
the Italian Ministry of Education and Research (MIUR). 
The first three authors are members of the research group GNAMPA of INdAM. 
\'Edouard Oudet gratefully acknowledges the support of the ANR through 
projects COMEDIC and OPTIFORM, and the support of the Labex Persyval Lab 
through project GeoSpec.

\section{Existence of solutions of the relaxed problem}
\label{sform}
Let us fix/recall the basic notation. 
Unless we specify otherwise, for the rest of this paper
$\Omega$ is a bounded Lipschitz domain in $\R^2$, the load 
$f$ is a signed measure on $\Omega$, the class $\AL$ of 
admissible reinforcements consists of all closed connected 
subsets of $\Omega$ with $\HH^1(S)\le L$.

For every $S\in\AL$ the functional $\EE_f(S)$ is defined by 
formula \eqref{energy}, with the linear term $\int_\Omega fu\,dx$ 
written as $\int_\Omega u\,df$ because $f$ is now a 
measure. The optimization problem \eqref{optpb} 
consists in finding the maximum of $\EE_f(S)$
among all $S\in\AL$, and it makes sense provided that 
$\EE_f$ is not identically $-\infty$ (see 
Subsection~\ref{conload}).

\smallskip
\emph{In the following we assume that $\EE_f(S)$ 
is finite for some set $S\in\AL$.} 

\smallskip
As explained in Subsection~\ref{relaxed}, 
we denote by $\M^+(\Omegabar)$ the class of all positive
finite measures on $\Omegabar$, and extend the 
$\EE_f$ to all $\mu$ in $\M^+(\Omegabar)$ 
by 
\begin{equation}
\label{emu}
\EE_f(\mu) := 
\inf \Big\{ 
	E_f(\mu,u) \colon u\in C^\infty_c(\Omega)
\Big\}
\, , 
\end{equation}
where 
\begin{equation}
\label{emuu}
E_f(\mu,u) := 
\frac{1}{2}\int_\Omega|\nabla u|^2\,dx
+ \frac{m}{2}\int_\Omega|\nabla u|^2\,d\mu
- \int_\Omega u\,df
\, .
\end{equation}
Then we denote by $\ML$ the weak* closure in $\M^+(\Omegabar)$ 
of the class of all measures of the form
$\HH^1\trace S$ with $S\in\AL$, 
in short
\[
\ML := 
\overline{ 
\big\{ \HH^1\trace S \colon S\in\AL \big\}
}^{\textstyle\text{\,weak*}}
\, ,
\]
and the relaxed optimization problem \eqref{maxbl} 
consists in finding the maximum of $\EE_f(\mu)$
among all $\mu\in\ML$.

\begin{proposition}
\label{closure}
The class $\ML$ consist of all 
$\mu\in\M^+(\Omegabar)$
such that
\begin{itemize}
\item[\rm(a)]
$\mu(\Omegabar)\le L$;
\item[\rm(b)]
the support $S$ of $\mu$ is a connected, compact set in $\Omegabar$
with $\HH^1(S)\le L$;
\item[\rm(c)]
$\HH^1\trace S \le \mu$.
\end{itemize}
\end{proposition}

We can now state our first existence result:

\begin{theorem}
\label{exthml}
The optimization problem~\eqref{maxbl}
admits a solution~$\mu\in\ML$.
\end{theorem}

Let $\mu$ be a measure in $\ML$ with support $S$.
We want to show that the value of
$\EE_f(\mu)$ is not affected if we replace
the full gradient $\nabla u$ in the integral 
$\int_\Omega |\nabla u|^2 d\mu$ in \eqref{emuu} with 
the tangential gradient $\nablatau u$, 
and if we remove from the measure $\mu$ 
the part which is singular with respect 
to~$\HH^1\trace S$.

\smallskip
To this purpose we need to recall some well-known 
properties of connected sets $S$ with finite length
(for more details we refer to standard references, 
such as~\cite{falconer}).

\subsection{Connected sets with finite length}
%
Let $S$ be a compact connected set 
in $\R^d$ with finite length. 
Then $S$ is rectifiable, 
which means that it can be covered 
(up to an $\HH^1$-negligible subset) by
countably many embedded curves of class $C^1$,
and indeed it can be parametrized (although not 
bijectively) by a single Lipschitz path.

Moreover $S$ admits a tangent line $\tau(x)$ at 
$\HH^1$-a.e.~$x\in S$. Thus, for every such $x$ and 
every $u:\R^d\to\R$ of class $C^1$ we can define the 
tangential gradient~$\nablatau u(x)$.

\subsection{The functionals $\boldsymbol{E^*_f}$ 
and $\boldsymbol{\EE^*_f}$ }
%
Let $\mu$ be a measure in $\M^+(\Omegabar)$ of the form
\[
\mu=\theta\,\HH^1\trace S
\]
where $S$ is a compact connected set in $\Omegabar$ 
with finite length.
We set
\begin{equation}
\label{emu*}
\EE^*_f(\mu) 
:= \inf\Big\{ 
	E^*_f(\mu,u) \colon u\in C^\infty_c(\Omega) 
\Big\}
\, ,
\end{equation}
where 
\[
E^*_f(\mu,u) := 
\frac{1}{2}\int_\Omega|\nabla u|^2\,dx
+ \frac{m}{2}\int_S |\nablatau u|^2\,\theta \, d\HH^1
- \int_\Omega u\,df
\, .
\]

\begin{theorem}
\label{equivalence}
Let $\mu$ be a measure in $\ML$ with support $S$,
and let $\mu^a$ be the absolutely continuous part of 
$\mu$ with respect to $\HH^1\trace S$.
Then 
\[
\EE_f(\mu)= \EE_f(\mu^a)= \EE_f^*(\mu^a)\,.
\] 
\end{theorem}

This statement shows that the maximum problem~\eqref{maxbl}
can be reformulated as the maximum of $\EE_f(\mu)$, or 
equivalently of $\smash{ \EE_f^*(\mu^a) }$, on the class
\begin{equation}
\label{mla}
\MLa
:= \Big\{
	\mu\in\ML \colon
	\text{$\mu$ is absolutely continuous w.r.t.~$\HH^1\trace \textrm{spt}(\mu)$}
\Big\}
\, .
\end{equation}
In view of Proposition~\ref{closure}, $\MLa$ agrees with the class 
of all measures of the form $\mu=\theta\,\HH^1\trace S$ where
\begin{itemize}
\item[\rm(a)]
$S$ is a compact connected set in $\Omegabar$ with $\HH^1(S)=L$, 
\item[\rm(b)]
$\theta$ is a real-valued multiplicity function with $\theta(x)\ge 1$ 
for $\HH^1$-a.e.~$x\in S$, 
\item[\rm(c)]
$\mu(\Omegabar)=\int_S \theta \, d\HH^1 \le L$.
\end{itemize}
The following improvement of Theorem~\ref{exthml} holds:

\begin{theorem}
\label{exthsl}
Problem~\eqref{maxbl} admits a solution $\mu$ 
in the class $\MLa$ with $\mu(\Omegabar)=L$, 
which is therefore a solution of
\[
\max\Big\{ 
	\EE_f^*(\mu) \colon \mu\in \MLa\, , \ \mu(\Omegabar)=L
\Big\}
\, .
\]
If in addition $f$ belongs to $L^p(\Omega)$ 
for some $p>1$ and the support of $f$ is $\Omegabar$, 
then \emph{every} solution of problem~\eqref{maxbl} 
belongs to $\ML^a$ and satisfies $\mu(\Omegabar)=L$. 
\end{theorem}

\begin{remark}
(i)~The first part of Theorem~\ref{exthsl} is
little more than a corollary of Theorem~\ref{equivalence}. 
The second part, however, requires a more delicate argument.

\smallskip
(ii)~The assumption that the support of $f$ is $\Omegabar$
in the second part of Theorem~\ref{exthsl} can probably 
be weakened, but not entirely removed.
Indeed if $f=0$ then $\EE_f(\mu)=0$ for
every $\smash{\mu\in\M^+(\Omegabar)}$, and in particular 
every $\mu\in\ML$ is a solution of problem~\eqref{maxbl}.

\smallskip
(iii)~As already pointed out in Subsection~\ref{relaxed}, 
if the solution $\mu=\theta\HH^1\trace S$ given by 
Theorem~\ref{exthsl} verifies $\theta=1$~a.e., that is, 
$\mu=\HH^1\trace S$, then $S$ is a solution of the original 
optimization problem~\eqref{optpb}. However, the following 
example suggests that this is not always the case.
\end{remark}

\begin{example}
\label{example}
Let $A,B$ be two points in $\Omega$ such that the closed segment $[A,B]$
is contained in $\Omega$, and let $f:=\delta_A-\delta_B$. 
Regarding the maximization problem~\eqref{maxbl} three possibilities 
may occur, depending on the choice of $L$:
\begin{itemize}
\item 
If $L<|A-B|$ then we have $\EE_f(S)=-\infty$ for every $S$ in the class $\AL$. 
Indeed, since a connected set $S$ of length $L<|A-B|$ cannot contain both 
$A$ and $B$, and since the capacity of a point in the plane is zero, 
we may construct a sequence of function $u_n \in C^\infty_c(\Omega)$ which 
vanish on $S$, tend to $0$ in $H^1(\Omega)$, and satisfy
$u_n(A)-u_n(B) \to +\infty$.
\item 
If $L=|A-B|$ then the unique set $S$ for which the energy is not $-\infty$ is the 
segment $[A,B]$, which is then the unique solution of the maximization 
problem~\eqref{optpb}, while $\mu:=\HH^1\trace S$ is the unique solution 
of problem~\eqref{maxbl}.
\item 
If $L>|A-B|$ and $\mu=\theta\,\HH^1\trace S$ is a solution 
of problem~\eqref{maxbl} as in Theorem~\ref{exthsl}, then the numerical 
simulations in Section \ref{snum} give a strong indication 
that $\theta>1$ on a subset of $S$ with positive length.
\end{itemize}
\end{example}

We conclude this section by defining the Sobolev space 
$H^1_0(\Omega)\cap H^1(S)$, and showing that
the infimum in formula \eqref{emu*} is actually 
a minimum.
This fact will be used in the proof of the second part
of Theorem~\ref{exthsl}.

\subsection{The Sobolev space $\boldsymbol{H^1(S)}$}
%
Let $S$ be a compact connected set 
in $\R^d$ with finite length $\ell:=\HH^1(S)$.
Using \cite[Theorem~4.4]{albott} we obtain 
a closed Lipschitz path $\gamma:[0,1]\to S$ 
which parametrizes~$S$, and more precisely
\begin{itemize}
\item
$\gamma$ has multiplicity $2$ at $\HH^1$-a.e.~$x\in S$, 
that is, $\#(\gamma^{-1}(x))=2$; 
\item
$|\dot\gamma(t)|=2\ell$ for a.e.~$t\in [0,1]$. 
\end{itemize}
We then define the Sobolev space $H^1(S)$ 
as the space of all continuous functions $u:S\to\R$
such that the composition $u\circ\gamma$ belongs 
to~$H^1(I)$, where $I:=(0,1)$.%
\footnote{
It can be proved that given a function $u$ in $H^1(S)$ 
and a Lipschitz function $\varphi:[0,1]\to S$
such that $|\dot\varphi(t)|\ge\delta$ for 
some given positive~$\delta$ and for a.e.~$t$, then 
$u\circ\varphi$ belongs to $H^1(I)$. 
In particular the definition of $H^1(S)$ does not depend
on the choice of the parametrization~$\gamma$.
}

Note that for a.e.~$t\in [0,1]$ the vector $\dot\gamma(t)$ spans 
$\tau(x)$, that is, the tangent line to $S$ at $x:=\gamma(t)$.
Therefore a function $u\in H^1(S)$ 
is differentiable along $\tau(x)$ for $\HH^1$-a.e.~$x\in S$ 
and the tangential derivative satisfies
\[
|\nablatau u(x)|=\frac{1}{2\ell} \big| (u\circ\gamma)'(t) \big|
\, .
\]
Thus the area formula for Lipschitz maps 
yields
\[
\big\|u\big\|_2^2
:=\int_S |u|^2 \, d\HH^1
=\ell\int_0^1 \big| (u\circ\gamma)'(t) \big|^2 \, dt
\, .
\]
and
\[
\big\|\nablatau u\big\|_2^2
:=\int_S |\nablatau u|^2 \, d\HH^1
=\frac{1}{4\ell} \int_0^1 \big| (u\circ\gamma)'(t) \big|^2 \, dt
\, .
\]
We endow $H^1(S)$ with the Hilbert norm
\[
\big\|u\big\|_{H^1(S)}^2 := \big\|u\big\|_2^2 + \big\|\nablatau u\big\|_2^2
\, .
\]

\subsection{The Sobolev space $\boldsymbol{H^1_0(\Omega)\cap H^1(S)}$}
%
Let $S$ be a connected compact set in $\Omegabar$
with finite length $\ell:=\HH^1(S)$.
Since $S$ is rectifiable, we can find a 
strictly positive (Borel) density function 
$m:S\to\R$ such that the trace operator
\[
T_S:H^1(\R^2) \to L^2(m\,\HH^1\trace S) 
\]
is well defined and bounded.%
\footnote{
Take for instance countably many 
compact curves $S_n$ of class $C^1$
in $\R^2$ that cover $\HH^1$-almost
all of $S$, let $m_n$ be the norm of the 
trace operators $T_n:H^1(\R^2) \to L^2(S_n)$, 
and set 
\[
m(x):=\sum_n \frac{1}{2^n m_n^2} 1_{S_n}(x)
\, .
\]}

Then we define $H^1_0(\Omega)\cap H^1(S)$
as the space of all $u\in H^1_0(\Omega)$
such that $T_Su$ 
agrees a.e.~with a function in $H^1(S)$.
In the following we tacitly assume that 
$u$ agrees on $S$ with 
the representative of the trace~$T_Su$ in $H^1(S)$, 
and in particular $u$ is continuous on~$S$.
We endow $H^1_0(\Omega)\cap H^1(S)$
with the Hilbert norm
\[
\big\| u \big\|_{H^1_0(\Omega)\cap H^1(S)}^2 
:= \big\| \nabla u \big\|_{L^2(\Omega)}^2 
+ \big\| \nablatau u \big\|_{L^2(S)}^2
\,.
\]
(Completeness can be proved using the continuity
of the trace operator~$T_S$.)
Using the fact that functions of class $H^{1/2}$ 
on intervals do not admit discontinuities of jump type, 
one can prove that for every $u$ in 
$H^1_0(\Omega)\cap H^1(S)$, there holds 
\[
u(x)=0
\quad\text{for every $x\in S\cap\bd\Omega$.}
\]

\begin{proposition}
\label{emumin}
Let $\mu$ be a measure of the form $\mu=\theta\,\HH^1\trace S$ 
where $S$ is a compact connected set with finite 
length in $\Omegabar$ and the multiplicity
$\theta$ is larger than some positive constant.
If $f$ belongs to $L^p(\Omega)$ for some 
$p>1$ then $\EE_f(\mu)=\EE_f^*(\mu)$ is finite and
\begin{equation}
\label{emu2}
\EE_f(\mu)
=\EE_f^*(\mu) 
= \min\Big\{
	E^*_f(\mu,u) \colon u\in H^1_0(\Omega)\cap H^1(S)
\Big\}
\, . 
\end{equation}
\end{proposition}

\section{Proofs of the results in Section~\ref{sform}}
\label{sproofs}
Through this section, given a matrix $M$ we denote by
$|M|$ the operator norm.

\begin{proposition}
\label{usc}
The functional $\EE_f$ defined in \eqref{emu}
is weakly* upper semicontinuous on $\M^+(\Omegabar)$.
\end{proposition}

\begin{proof}
Just notice that $\EE_f(\mu)$ is defined as the infimum 
of $E_f(\mu,u)$ over all $u$ in $C^\infty_c(\Omega)$
and $E_f(\mu,u)$ is clearly weakly* continuous in $\mu$
for every such~$u$, cf.~\eqref{emuu}.
\end{proof}

\begin{proof}[Proof of Theorem \ref{exthml}]
This statement follows from Proposition~\ref{usc} 
and the weak* compactness of the class $\ML$ 
(an immediate consequence of its definition).
\end{proof}

To prove Proposition~\ref{closure} we need the following results,
which we state in $\R^d$ even if we need only the case~$d=2$.


\begin{proposition}\label{similgolab}
Let $(S_n)$ be a sequence of compact connected sets 
in $\R^d$ which converge to a set $S$ in the Hausdorff distance,
let $(\mu_n)$ be a sequence of positive finite measures on $\R^d$ 
which converge to a measure $\mu$ in the weak* sense, 
and assume that
\begin{itemize}
\item
$\mu_n$ is supported on $S_n$ and $\HH^1\trace S_n \le \mu_n$;
\item
$\HH^1(S_n)\le L$ for some finite constant $L$.
\end{itemize}
Then
\begin{itemize}
\item
$\mu$ is supported on $S$ and $\HH^1\trace S \le \mu$;
\item
$\HH^1(S)\le \liminf \HH^1(S_n) \le L$.
\end{itemize}
\end{proposition}

\begin{proof}
The fact that $\mu$ is supported on $S$ follows easily 
from the weak* convergence of $\mu_n$ to $\mu$ and 
the convergence of $S_n$ to $S$ in the Hausdorff distance. 

The inequality $\HH^1(S)\le \liminf \HH^1(S_n)$ is 
\Golab's semicontinuity theorem (see \cite[Section~3]{golab}, 
or \cite[Theorem~2.9]{albott}). 

The inequality $\HH^1\trace S \le \mu$ can be viewed as a 
localized version of \Golab's theorem, 
and the proof is slightly more complicated.
Using \cite[Theorem~4.4]{albott}, we obtain
that each $S_n$ can be parametrized 
by a closed path $\gamma_n:[0,1]\to S_n$ such that
\begin{itemize}
\item
$\gamma_n$ has multiplicity $2$ at $\HH^1$-a.e.~$x\in S_n$, 
that is, $\#(\gamma_n^{-1}(x))=2$;
\item
$\gamma_n$ has degree $0$ at $\HH^1$-a.e.~$x\in S_n$
(for a precise definition see \cite[{\S 4.1}]{albott});
\item
each $\gamma_n$ has Lipschitz constant $\Lip(\gamma_n)\le 2L$.
\end{itemize}
Passing to a subsequence we can assume that the paths
$\gamma_n$ converge uniformly to some $\gamma:[0,1]\to S$ 
with $\Lip(\gamma)\le 2L$, and one easily checks that 
$\gamma$ parametrizes~$S$. Moreover 
\cite[Proposition~4.3]{albott} shows that $\gamma$ has 
degree $0$ and in particular has multiplicity at least 2 
at $\HH^1$-a.e.~$x\in S$.

Therefore, for every \emph{positive} test function 
$\varphi\in C_c(\R^d)$ 
there holds
\begin{align*}
\int_{\R^d} \varphi\, d\mu
=\lim_{n\to\infty} \int_{\R^d} \varphi\, d\mu_n
& \ge\liminf_{n\to\infty} \int_{S_n} \varphi\, d\HH^1 \\
& =\liminf_{n\to\infty} \frac{1}{2} \int_0^1 \varphi(\gamma_n)\, |\dot\gamma_n|\, dt \\
& \ge \frac{1}{2} \int_0^1 \varphi(\gamma)\, |\dot\gamma|\, dt
  \ge \int_S \varphi\, d\HH^1 
  \, .
\end{align*}
The first inequality follows from the assumption $\mu_n\ge \HH^1\trace S_n$
and the fact that $\varphi$ is positive, 
the second equality and the third inequality follow from the area formula
for Lipschitz maps, 
and finally the second inequality follows by a standard semicontinuity 
argument.

We have thus proved that
\[
\int_{\R^d}\varphi\,d\mu\ge\int_S \varphi\,d\HH^1
\]
for every test function $\varphi\ge0$, which yields the desired inequality
$\mu\ge\HH^1\trace S$.
\end{proof}

\begin{lemma}\label{approx}
Let $S$ be a connected compact set with finite length in $\R^d$,
and let $\mu$ be a positive finite measure supported on $S$
such that $\mu\ge\HH^1\trace S$.
Then there exists a sequence of connected compacts sets $S_n$ in $\R^d$ 
such that
\begin{itemize}
\item
$\HH^1(S_n)\le\mu(\R^d)$ for every $n$, 
\item
the sets $S_n$ converge to $S$ in the Hausdorff distance;
\item
the measures $\HH^1\trace S_n$ converge weakly* to $\mu$.
\end{itemize}
\end{lemma}

\begin{proof}
Choose a unit vector $e\in\R^d$ which is not in the 
approximate tangent line to $S$ at $x$ for $\HH^1$-a.e.~$x\in S$.
Consider then the measure $\lambda:=\mu-\HH^1\trace S$.
Since $\lambda$ is positive and supported on $S$, 
it can be approximated by a sequence of positive 
discrete measures
\[
\lambda_n := \sum_j a_{nj} \delta_{x_{nj}}
\]
where
\begin{itemize}
\item
the points $x_{nj}$ belong to $S$,
\item
the coefficients $a_{nj}$ converge 
uniformly to $0$ as $n\to\infty$,
\item
$\lambda_n(\R^d) \le \lambda(\R^d)$ for
every $n$, that is, $\sum_j a_{nj} \le \mu(\R^d)-\HH^1(S)$.
\end{itemize}
Thanks to the choice of $e$ we can further require that
\begin{itemize}
\item
$x_{nj}-x_{ni}$ is not parallel to $e$ for every $n$ and every $i\ne j$.
\end{itemize}
Finally we set
\[
S_n := S\cup \Big( \bigcup_j I_{nj} \Big)
\]
where $I_{nj}$ is the closed segment with endpoints $x_{nj}$ and $x_{nj}+a_{nj}e$.
By the choice of the points $x_{ni}$ we have that the segments $I_{nj}$ 
are pairwise disjoint (for fixed $n$) and have negligible intersection with $S$.
Now it is easy to check that the sets $S_n$ have the required properties.
\end{proof}

\begin{proof}[Proof of Proposition~\ref{closure}]
We first prove that every measure $\mu\in\ML$
satisfies properties (a)--(c) in Proposition~\ref{closure}.
Take indeed a sequence of sets $S_n\in\AL$ such that
$\HH^1\trace S_n$ weakly* converge to $\mu$.
Possibly passing to a subsequence we can assume
that the sets $S_n$ converge in Hausdorff distance 
to some compact connected set $S$ contained 
in $\Omegabar$.
Then Proposition~\ref{similgolab} implies that $S$ 
is the support of $\mu$ and $\mu$ belongs to $\ML$. 

The converse implication, namely that every positive 
measure that satisfies properties (a)-(c) 
belongs to $\ML$, follows from Lemma~\ref{approx}.
\end{proof}

The proof of Theorem \ref{equivalence} 
is split in two parts (Propositions~\ref{killsing} 
and~\ref{killnorm}),
the proofs of which require several lemmas.
Some of these lemmas are stated 
in general dimension~$d$ even if we only need the case~$d=2$.

\begin{lemma}
\label{effe1}
Let $K\subset\R$ be a compact set with $|K|=0$. 
For every $\eps>0$ there exists a function $f:\R\to\R$ 
of class $C^\infty(\R)$
such that
\begin{itemize}
\item
$|f(x)-x|\le\eps$ for all $x\in\R$;
\item 
$f'=0$ in a neighborhood of $K$; 
\item
$0\le f'(x)\le 1$ for all $x\in\R$;
\item 
the open set $A:=\{x\colon f'(x) \ne 1\}$ 
satisfies $|A| \le \eps$.
\end{itemize}
\end{lemma}

\begin{proof}
For every $\delta>0$ let
\begin{itemize}
\item
$K_\delta$ be the open $\delta$-neighborhood of $K$; 
\item
$\rho_\delta$ be a smooth regularizing kernel 
with support contained in $[-\delta,\delta]$, 
\item
$g_\delta:= \rho_\delta * 1_{K_{2\delta}}$ and $f_\delta$ 
be a primitive of $1-g_\delta$ such that $f_\delta(0)=0$.
\end{itemize}
One easily checks that $g_\delta$ and $f_\delta$ 
are smooth functions and satisfy the following properties:
\begin{itemize}
\item
$0\le g_\delta(x) \le 1$ for every $x\in\R$, which implies $0 \le f_\delta'(x) \le 1$;
\item
if $x\in K_\delta$ then $[x-\delta, x+\delta]\subset K_{2\delta}$, hence
$g_\delta(x)=1$, and $f_\delta'(x)=0$;
\item
if $x\notin K_{3\delta}$ then $[x-\delta, x+\delta]\cap K_{2\delta}=\varnothing$, 
hence $g_\delta(x)=0$ and $f_\delta'(x)=1$;
\item
for every $x\in\R$, $|f_\delta(x)-x| \le \|g_\delta\|_1=|K_{2\delta}|$.
\end{itemize}
Notice now that since $K$ is compact then $|K_\delta|$ converges 
to $|K|=0$ as $\delta\to 0$, and therefore we can find $\bar\delta$ 
such that $|K_{2\bar\delta}| \le |K_{3\bar\delta}|\le\eps$.
We conclude the proof by setting $f:=f_{\bar\delta}$.
\end{proof}

\begin{lemma}
\label{effe2} 
Let $K$ be a compact set in $\R^d$, $d\ge 2$, with $\HH^1(K)=0$ 
For every $\eps>0$ there exist a map $\phi:\R^d\rightarrow\R^d$ 
of class $C^\infty$ and an open set $A\subset\R^d$ such that
\begin{itemize}
\item[\rm(i)] 
$|\phi(x)-x|\le\eps$ for all $x\in\R^d$;
\item[\rm(ii)]
$\nabla\phi=0$ in a neighborhood of $K$;
\item[\rm(iii)]
$|\nabla\phi(x)|\le 1$ for all $x\in\R^d$.
\end{itemize}
Moreover, having fixed $r>0$, we can further require that 
$\nabla\phi(x)=I$, where $I$ is the $d\times d$-identity matrix, 
out of an open set $A$ with $|A\cap(-r, r)^d|\le\eps$.
\end{lemma}

\begin{proof}
Fix for the time being $\eps'$, to be properly chosen later.
For $i=1,\dots, d$ let $K_i$ be the projection of $K$ on the 
$i$-th coordinate axis, and let $f_i:\R\to\R$ be 
the function obtained by applying Lemma~\ref{effe1} 
with $K_i$ and $\eps'$ in place of $K$ and $\eps$, and let $A_i$
be the set where $f_i'\ne 1$. 
Let
\[
\phi(x_1,\dots,x_d):= \big(f_1(x_1),\dots,f_d(x_d)\big)
\, .
\]
It is easy to check that $\phi$ has properties (i)-(iii) 
for $\eps'$ small enough.
Moreover $A$ is contained in the set of all 
$x$ such that $x_i\in A_i$ for some $i=1,\dots,d$, 
and therefore $|A\cap (-r,r)^d| \le d\eps'(2r)^{d-1}$, 
which is less that $\eps$ for $\eps'$ small enough. 
\end{proof}

\begin{lemma}
\label{killsinglemma}
Let $\mu, \mu'$ be measures in $\M^+(\Omegabar)$
and assume that $\mu=\mu'+\lambda$ where $\lambda$ is a 
positive measure supported on a 
Borel set $E$ with $\HH^1(E)=0$.
Then for every $u\in C^\infty_c(\Omega)$ 
and every $\delta>0$ there exist $v\in C^\infty_c(\Omega)$
such that 
$\|v-u\|_\infty \le\delta$ and
\begin{equation}
\label{killsingestimate}
\int_\Omega |\nabla v|^2 d\mu 
\le \int_\Omega |\nabla u|^2 d\mu' + \delta
\, . 
\end{equation}
\end{lemma}

\begin{proof}
We fix for the time being $\eps>0$,
to be chosen properly through the proof. 
Then we choose a compact set $K\subset E$ such that
$\lambda(\Omegabar\setminus K)\le\eps$, 
we let $\phi:\R^2\to\R^2$ be the map
constructed in Lemma~\ref{effe2} for the set $K$,
and we set
\[
v(x):=u(\phi(x))
\quad\text{for every $x\in\R^2$.}
\]
The function $v$ is clearly smooth and compactly supported 
on $\R^2$, and the support is contained in $\Omega$ for 
$\eps$ sufficiently small.

The rest of the proof is divided in three steps.
In the following we use the letter $C$ 
to denote any constant that may depend on $\mu$ and $u$ 
but not on $\eps$; the value of $C$ may change
at every occurrence.

\smallskip
\emph{Step~1. Estimate of $|u-v|$.} 
Using property~(i) in Lemma~\ref{effe2}
we obtain
\[
|v-u| = |u(\phi)-u | 
\le \Lip(u) \, |\phi(x)-x|
\le C\eps
\, , 
\]
which implies $\|v-u\|_\infty\le\delta$ if we choose 
$\eps$ small enough.

\smallskip
\emph{Step~2. Estimates of $|\nabla v|$.} 
Using property~(iii) in Lemma~\ref{effe2}
we obtain
\begin{align*}
  |\nabla v|
  = |\nabla u(\phi)| \, |\nabla\phi|
  \le |\nabla u(\phi)| 
& \le |\nabla u| + \big|\nabla u(\phi) - \nabla u\big| \\
& \le |\nabla u| + \Lip(\nabla u) \, \eps
  = |\nabla u| + C\eps
  \,.
\end{align*}

\emph{Step~3. Proof of estimate~\eqref{killsingestimate}.} 
Property~(ii) in Lemma~\ref{effe2} implies that $\nabla v=0$ on $K$.
Using this fact and the estimate in Step~2, and recalling the 
choice of $K$, we obtain
\begin{align*}
\int_\Omega |\nabla v|^2 d\mu
= \int_{\Omega\setminus K} |\nabla v|^2 d\mu
& \le \int_{\Omega\setminus K} 
       \big( |\nabla u|+ C\eps\big)^2 \, d\mu \\
& = \int_{\Omega\setminus K} |\nabla u|^2 + C\eps \, d\mu \\
& \le \int_\Omega |\nabla u|^2 \, d\mu' 
  + \int_{\Omega\setminus K} |\nabla u|^2 \, d\lambda
  + C\eps \\
& \le \int_\Omega |\nabla u|^2 \, d\mu' 
  + C\eps
  \, ,
\end{align*}
which implies the desired estimate if we choose 
$\eps$ small enough.
\end{proof}

\begin{lemma}
\label{killsinglemma2}
Let $\mu, \mu'$ be as in Lemma~\ref{killsinglemma}.
Then for every $u\in C^\infty_c(\Omega)$ 
and every $\delta>0$ there exist $v\in C^\infty_c(\Omega)$
such that $\|v-u\|_\infty \le\delta$ and
\begin{equation}
\label{killsingestimate2}
E_f(\mu,v) \le E_f(\mu',u) + \delta
\, . 
\end{equation}
\end{lemma}

\begin{proof}
Consider the measures $\bar\mu:=dx + m\mu$
and $\bar\mu':=dx + m\mu'$, where $dx$ is the Lebesgue
measure on $\Omega$, and fix $\bar\delta \in (0,\delta]$, to
be chosen later.

Let now $v$ be the function obtained 
by applying Lemma~\ref{killsinglemma} with $\bar\mu$, $\bar\mu'$, 
$\bar\delta$ in place of $\mu$, $\mu'$, $\delta$.
Then $\|v-u\|_\infty \le\bar\delta \le\delta$ and
\begin{align*}
  E_f(\mu,v)  
& = \frac{1}{2}\int_\Omega |\nabla v|^2 d\bar\mu + \int_\Omega v\,df \\
& \le \frac{1}{2}\bigg[ \int_\Omega |\nabla u|^2 d\bar\mu' + \bar\delta \bigg]
    + \bigg[ \int_\Omega u\,df + \|u-v\|_\infty \|f\| \bigg] \\
& = E_f(\mu',u) + \Big( \frac{1}{2}+\|f\| \Big) \, \bar\delta 
  \, ,
\end{align*}
where $\|f\|$ is the total mass of the signed measure $f$.
Thus estimate~\eqref{killsingestimate2}
follows by choosing $\bar\delta$ sufficiently small.
\end{proof}

\begin{lemma}
\label{nonsharp}
Given $\mu,\mu'\in\M^+(\Omegabar)$ such that $\mu\ge\mu'$, 
then $\EE_f(\mu) \ge \EE_f(\mu')$.
\end{lemma}

\begin{proof}
This statement follows by the fact that $E_f(\mu',u) \le E_f(\mu,u)$
for every $u$ in $C^\infty_c(\Omega)$,
cf.~\eqref{emu} and~\eqref{emuu}.
\end{proof}

\begin{proposition}
\label{killsing}
Let $\mu, \mu'$ be measures in $\M^+(\Omegabar)$
and assume that $\mu=\mu'+\lambda$ where $\lambda$ is a 
positive measure supported on a 
Borel set $E$ with $\HH^1(E)=0$.
Then
\[
\EE_f(\mu) = \EE_f(\mu')
\, .
\]
\end{proposition}

\begin{proof}
The inequality $\EE_f(\mu) \ge \EE_f(\mu')$
is contained in Lemma~\ref{nonsharp}, 
the opposite inequality follows
from Lemma~\ref{killsinglemma2}. 
\end{proof}

\begin{lemma}
\label{effe3}
Let $\mu$ be a finite positive measure on $\R^d$
of the form $\mu=\theta\,\HH^1\trace\Sigma$ where
$\Sigma$ is a rectifiable set, 
and let $A$ be an open set that contains $\Sigma$.
For a.e.~$x\in\Sigma$ let $\tau(x)$ be the tangent 
line to $\Sigma$ at $x$ and let $P(x)$ be the matrix associated to 
the orthogonal projection of $\R^d$ onto $\tau(x)$.
Then for every $\eps>0$ there exist a smooth map $\phi:\R^d\to\R^d$
and a compact set $K\subset\Sigma$ such that
\begin{itemize}
\item[\rm(i)]
$|\phi(x)-x|\le\eps$ for all $x$, 
and $\phi(x)=x$ for $x\in\R^d\setminus A$;
\item[\rm(ii)]
$|\nabla\phi(x)|\le 10$ for all $x$, and $\nabla\phi(x)=I$
for $x\in\R^d\setminus A$;
\item[\rm(iii)]
$|\nabla\phi(x)-P(x)|\le \eps$ for all $x\in K$ and
$\mu(\R^d\setminus K)\le\eps$.
\end{itemize}
\end{lemma}

\begin{proof}
The proof is divided in several steps.
We fix for the time being $\eps'\in (0,1]$, 
to be chosen at the end of the proof.

\smallskip
\emph{Step~1. Construction of $\phi$.}
Using the fact that $\Sigma$ is rectifiable 
and $\mu$ is supported on $\Sigma$ we can 
find compact sets $K_i\subset\Sigma$ 
and curves $\Gamma_i$ of class $C^1$, with
$i=1,\dots,n$, such that:
\begin{itemize}
\item
the sets $K_i$ are disjoint and contained in $\Gamma_i$;
\item
$\mu(\R^d\setminus K)\le\eps$ where $K:=K_1\cup\dots\cup K_n$;
\item
up to a rotation $R_i$, $\Gamma_i$ 
agrees with the graph of a $C^1$ map $\gamma_i:\R\to\R^{d-1}$ 
(we identify $\R^d$ with $\R\times\R^{d-1}$) 
and $|\dot\gamma_i|\le\eps'$ everywhere.
\end{itemize}
Then we find $\delta\in (0,\eps/3]$ such that
\begin{itemize}
\item
the open $\delta$-neighborhoods $K_i^\delta$ are disjoint 
and contained in~$A$.
\end{itemize}
Next we find smooth curves $G_i$ 
and smooth functions $\sigma_i:\R^d\to [0,1]$ such that
\begin{itemize}
\item
$\sigma_i=1$ on a neighborhood of $K_i$ and $\sigma_i=0$ 
out of $K_i^\delta$;
\item
$|\nabla\sigma_i|\le 2/\delta$ everywhere;
\item
$G_i$ agrees, up to the rotation $R_i$, with the graph of a 
$C^\infty$ map $g_i:\R\to\R^{d-1}$ such that
$|\gamma_i-g_i| \le \delta$ and $|\dot g_i| \le \eps'$ everywhere.
\end{itemize}
For every $i$ we let $p_i$ be the projection 
of $\R^n$ onto $G_i$ defined by
\[
p_i(x)=p_i(x_1,\dots,x_d):= \big( x_1,g_i(x_1) \big)
\]
(modulo the rotation $R_i$).
Finally we set $\sigma_0:= 1 - (\sigma_1 +\cdots+\sigma_n)$, 
so that the functions $\sigma_0,\dots,\sigma_n$ form a 
partition of unity, and define
\begin{equation}
\label{formulafi}
\phi(x) 
:= \sigma_0(x) \, x + \sum_{i=1}^n \sigma_i(x) \, p_i(x)
= x + \sum_{i=1}^n \sigma_i(x) \, \big( p_i(x)-x \big)
\, .
\end{equation}

In the rest of this proof we assume for simplicity 
that the rotations $R_i$ are the identity, and 
we denote by $C$ any constant that does not depend 
on $\eps'$ and $\delta$; the value of $C$ may 
vary at every occurrence.

\smallskip
\emph{Step~2. $|x-p_i(x)|\le 3\delta$ for every $x\in K^\delta_i$.}
Write $x=(x_1,x')$ with $x'\in\R^{d-1}$.
Since $\Gamma_i$ is the graph of $\gamma_i$ and $|\dot\gamma_i|\le 1$,
the fact that $\dist(x,\Gamma_i)\le\delta$ implies
$|x'-\gamma_i(x_1)| \le 2\delta$. Then the assumption $|\gamma_i-g_i|\le\delta$ 
yields $|x'-g_i(x_1)| \le 3\delta$, which is the desired estimate.

\smallskip
\emph{Step~3. Proof of statement~(i).}
Given $x\in\R^n$, for every $i$ such that $\sigma_i(x)\ne 0$
there holds $x\in K^\delta_i$, and then 
$|x-p_i(x)| \le 3\delta$ by Step~2. Hence~\eqref{formulafi} gives
\[
|\phi(x)-x| 
\le \sum_{i=1}^n \sigma_i(x) \, \big| p_i(x)-x \big| 
\le 3\delta \le\eps
\, . 
\]
On the other hand, if $x$ does not belong 
to $A$ then it does not belong to any $K^\delta_i$, 
which means $\sigma_i(x)=0$, and therefore \eqref{formulafi} 
yields $\phi(x)=x$.

\smallskip
\emph{Step~4. Proof of statement~(ii).}
Formula~\eqref{formulafi} gives
\begin{equation}
\label{nablafi}
\nabla\phi(x) 
= I +\sum_{i=1}^n \sigma_i(x) \big(\nabla p_i(x) - I\big) 
    +\sum_{i=1}^n \big(p_i(x) - x\big) \otimes \nabla\sigma_i(x) 
\,. 
\end{equation}
If $x$ does not belong to $A$ then $x$ does not belong
to the support of $\sigma_i$ for every $i$, and formula~\eqref{nablafi} 
reduces to $\nabla\phi(x)=I$.

Consider now $x$ arbitrary. From formula~\eqref{nablafi} we obtain
\begin{align*}
|\nabla\phi(x)|
& \le 1 +\sum_{i=1}^n \sigma_i(x) \, \big(1+|\nabla p_i(x)|\big)
        +\sum_{i=1}^n \big|x-p_i(x)\big| \cdot \big|\nabla\sigma_i(x)\big| \\
& \le 1 + 3 + 3\delta \cdot \frac{2}{\delta} = 10
  \, .
\end{align*}
For the second inequality we used that 
$\sigma_i(x)= 0$ and $\nabla\sigma_i(x)= 0$
for all $i$ except at most one, and the following estimates:
$|\nabla p_i(x)|\le 2$ (use the definition of $p_i$ 
and the bound $|\dot g_i|\le\eps'\le 1$), 
$|\nabla\sigma_i|\le 2/\delta$ (by the choice of $\sigma_i$), 
and $|x-p_i(x)|\le 3\delta$ (Step~2). 

\smallskip
\emph{Step~5. $|\nabla p_i(x)-P(x)|\le C\eps'$ for $x\in K_i$.}
Let $P$ be the $d\times d$ matrix associated to the projection of
$\R^d$ onto the line $\R\times\{0\}$, that is, the matrix with
all entries equal to $0$ except $P_{11}=1$.
The definition of $p_i$ and the assumption $|\dot g_i|\le\eps'$ imply
$|\nabla p_i(x)-P|\le C\eps'$, while the assumption $|\dot\gamma_i|\le\eps'$
implies $|P(x)-P|\le C\eps'$. 

\smallskip
\emph{Step~6. Proof of statement~(iii).}
We already know that $\mu(\R^d\setminus K)\le\eps$.
Moreover, if $x$ belongs to $K$ then it belongs to $K_i$ 
for some $i$ and since $\sigma_i=1$ in a neighborhood of $K_i$, 
formula~\eqref{nablafi} reduce to $\nabla\phi(x)=\nabla p_i(x)$. 
We conclude the proof using the estimate in Step~5 
and choosing $\eps'$ small enough.
\end{proof}

\begin{lemma}
\label{killnorlemma}
Let $\mu$ be a measure in $\M^+(\Omegabar)$
of the form $\mu=\theta\,\HH^1\trace\Sigma$ where
$\Sigma$ is a rectifiable set.
Then for every $u\in C^\infty_c(\Omega)$ 
and every $\delta>0$ there exist $v\in C^\infty_c(\Omega)$
such that $\|v-u\|_\infty \le\delta$ and
\begin{equation}
\label{killnorestimate}
E_f(\mu,v) \le E_f^*(\mu,u) + \delta
\, .
\end{equation}
\end{lemma}

\begin{proof}
We fix $\eps>0$, to be chosen later. 
We take an open set $A\supset\Sigma$ 
such that $|A|\le\eps$, 
and we then let $\phi$ and $K$ be the map 
and the compact set given by Lemma~\ref{effe3}.
We now set
\[
v(x):=u(\phi(x))
\quad\text{for every $x\in\R^2$.}
\]
The function $v$ is smooth and its support is 
contained in $\Omega$ for $\eps$ sufficiently small, 
and the desired properties of $v$ follow by the properties of~$\phi$ 
stated in Lemma~\ref{effe3}. As usual, the letter $C$ denotes
any constant which does not depend on $\eps$.

Using property~(i) in Lemma~\ref{effe3}, 
for every $x\in\R^2$ we obtain
\[
|v-u| 
= |u(\phi)-u| 
\le \Lip(u) \, |\phi(x)-x| 
\le C\eps
\, , 
\]
which implies $\|v-u\|_\infty\le\delta$ for $\eps$ small enough, 
and
\begin{equation}
\label{killnor2}
\bigg| \int_\Omega v\, df - \int_\Omega u\, df \bigg|
\le C\eps 
\, .
\end{equation}
Using property~(ii) in Lemma~\ref{effe3} we obtain
\begin{equation}
\label{killnor3}
|\nabla v|=|\nabla u(\phi) | \cdot |\nabla\phi|
\le C 
\, ,
\end{equation}
while properties~(i) and (ii) imply that
$\nabla v=\nabla u$ for every $x\in\Omega\setminus A$.
Therefore
\begin{equation}
\label{killnor4}
\int_\Omega |\nabla v|^2 \, dx 
\le \int_{\Omega\setminus A} |\nabla u|^2 \, dx + C|A|
\le \int_\Omega |\nabla u|^2 \, dx + C\eps
\,.
\end{equation}
For every $x\in K$ we write $\nabla v$ as follows, 
where $P(x)$ is the matrix associated to the projection 
on the approximate tangent line to $\Sigma$ at $x$:
\[
\nabla v 
= \nabla u(\phi)\,\nabla\phi
= \big( \nabla u(\phi) -\nabla u\big)\,\nabla\phi
 +\nabla u\,\big(\nabla\phi-P\big)+\nabla u\,P
\, .
\]
Then, recalling that $\nabla u \, P=\nablatau u$
and using properties (i)-(iii), we obtain:
\begin{align*}
|\nabla v| 
& \le \big| \nabla u(\phi) -\nabla u\big|\cdot|\nabla\phi|
      + |\nabla u| \cdot |\nabla\phi-P| + |\nablatau u| \\
& \le \Lip(\nabla u) \, |\phi(x)-x| \, C
      + C\eps + |\nablatau u|
   =|\nablatau u| +C\eps
   \,.
\end{align*}
Using this estimate and \eqref{killnor3}, and the fact that
$\mu(\R^2\setminus K)\le\eps$, we obtain
\begin{equation}
\label{killnor5}
\int_{\Sigma} |\nabla v|^2 \, d\mu 
= \int_K |\nablatau u|^2 +C\eps \, d\mu 
  + \int_{\Sigma\setminus K} C \, d\mu
\le \int_\Sigma |\nablatau u|^2 \, d\mu + C\eps
\, .
\end{equation}
Finally \eqref{killnorestimate}
follows from \eqref{killnor2}, \eqref{killnor4} 
and \eqref{killnor5} by choosing $\eps$ small enough.
\end{proof}

\begin{proposition}
\label{killnorm}
Let $\mu$ be a measure in $\M^+(\Omegabar)$
of the form $\mu=\theta\,\HH^1\trace\Sigma$ where
$\Sigma$ is a rectifiable set.
Then
\[
\EE_f(\mu) = \EE_f^*(\mu)
\, .
\]
\end{proposition}

\begin{proof}
The trivial inequality $E_f(\mu,u) \ge E_f^*(\mu,u)$
implies $\EE_f(\mu) \ge \EE_f^*(\mu)$; the opposite 
inequality follows from Lemma~\ref{killnorlemma}.
\end{proof}

\begin{proof}[Proof of Theorem~\ref{equivalence}] 
Combine Propositions~\ref{killsing} and~\ref{killnorm}.
\end{proof}

\begin{proof}[Proof of Proposition~\ref{emumin}] 
If $f$ belongs to $L^p(\Omega)$ for some $p>1$ then it 
belongs also to the dual of $H^1_0(\Omega)$, and therefore 
the functional $\smash{E_f^*(\mu,\cdot)}$ is coercive, 
which implies that the infimum $\EE_f^*(\mu)$ is not $-\infty$. 
Clearly the same holds for $\EE_f(\mu)$.

Next we notice that the minimum of $\smash{E_f^*(\mu,u)}$ 
over all $u\in\smash{H^1_0(\Omega)\cap H^1(S)}$ is attained
because this functional is coercive and weakly 
lower-semicontinuous.

The first equality in \eqref{emu2} is proved in 
Theorem~\ref{equivalence}.

To prove the second equality in \eqref{emu2}
it is enough to show that $\smash{C^\infty_c(\Omega)}$
is dense in norm in $\smash{H^1_0(\Omega)\cap H^1(S)}$.
The proof of this density result is a bit 
delicate, but ultimately standard, 
and we simply list the key steps:
\begin{itemize}
\item
$C^\infty_c(\Omega)$ is dense in $\Lip_c(\Omega)$;
\item
$\Lip_c(\Omega)$ is dense in $\Lip_0(\Omega)$;
\item
$\Lip_0(\Omega)$ is dense in the subspace $X$ 
of all $u\in H^1_0(\Omega)\cap H^1(S)$ which are
constant on some open set $A$ (depending 
on $u$) such that $S\setminus A$ can be covered
by finitely many disjoint compact curves of 
class $C^1$;
\item
$X$ is dense in $\smash{H^1_0(\Omega)\cap H^1(S)}$. 
\end{itemize}
In all these statement ``dense'' refers to 
the norm of $H^1_0(\Omega)\cap H^1(S)$;
the last statement is the most delicate, 
and can be proved arguing as in the proof 
of Lemma~\ref{killsinglemma}.
\end{proof}

\begin{lemma}
\label{sharp}
Assume that $f\in L^p(\Omega)$ for some $p>1$
and that the support of $f$ is $\Omega$, 
and let $\mu$ be a measure in $\ML$ such that 
$\mu(\Omegabar)<L$.
Then there exists $\mu'$ in $\ML$ such that
$\mu'\ge\mu$ and $\EE_f(\mu')>\EE_f(\mu)$.

In particular every solution $\mu$ of problem~\eqref{maxbl}
satisfies $\mu(\Omegabar)=L$.
\end{lemma}

\begin{proof}
Let $S$ be the support of $\mu$. By Proposition~\ref{emumin}, 
$\EE_f(\mu)=E^*_f(\mu,u)$ where $u\in H^1_0(\Omega)\cap H^1(S)$ 
is a minimizer of $E^*_f(\mu,\cdot)$.
Then $u$ solves the equation 
$\Delta u=-f$ in $\Omega':=\Omega\setminus S$, 
which implies that $u$ is of class $C^1$ on 
$\Omega'$ and the set of all $x\in\Omega'$
such that $\nabla u(x)=0$
has empty interior.

In particular we can find a point $x\in\Omega'$
such that $\nabla u(x)\ne 0$ and 
$\dist(x,S)<\ell$ where $\ell:=L-\mu(\Omegabar)$.
We then choose a segment $S'$ which connects $x$ to $S$,
has length $\HH^1(S')\le \ell$, and
is not orthogonal to $\nabla u(x)$. 

We set $\mu':=\mu+\HH^1\trace S$.
Clearly $\mu'\ge\mu$ and the support of $\mu'$ is $S\cup S'$, 
and one easily checks that $\mu'$ belongs to $\ML$. 
Since $\mu'\ge\mu$ then $\EE_f(\mu')\ge\EE_f(\mu)$
(cf.~Lemma~\ref{nonsharp}), 
and we claim that this inequality is strict.

\smallskip
Assume by contradiction that $\smash{\EE_f(\mu')=\EE_f(\mu)}$, 
and let $u'\in H^1_0(\Omega)\cap H^1(S\cup S')$ 
be a minimizer of $\smash{E^*_f(\mu',\cdot)}$.
Then $u'$ is also a minimizer of $\smash{E^*_f(\mu,\cdot)}$, 
and since this functional is strictly convex
we have that $u$ and $u'$ agree as elements of
the space $H^1_0(\Omega)\cap H^1(S)$. This means that
\[
E^*_f(\mu,u') = \EE_f(\mu) = \EE_f(\mu') = E^*_f(\mu',u') 
\, .
\]
On the other hand $u$ is of class $C^1$ on $\Omega\setminus S$, 
and in particular is continuous, 
and therefore $u$ agrees with $u'$ on $S'$, which implies that
$\nablatau u'=\nablatau u$ a.e.~on $S'$, and by the choice 
of $S'$ we have that $\nablatau u$ is not identically null on $S'$.
This yields the contradiction $E^*_f(\mu,u')<E^*_f(\mu',u')$.
\end{proof}

\begin{proof}[Proof of Theorem \ref{exthsl}]
Let us prove the first part of the statement.
Let $\bar\mu\in\ML$ be an arbitrary solution of problem~\eqref{maxbl} 
(which exists by Theorem~\ref{exthml}), 
let $S$ be the support of $\bar\mu$ and let $\bar\mu^a=\bar\theta\,\HH^1\trace S$ 
be the absolutely continuous part of $\bar\mu$ with respect to $\HH^1\trace S$. 
Then $\bar\mu^a$ is also a solution of problem~\eqref{maxbl} by 
Theorem~\ref{equivalence}.

\smallskip
If $L=\bar\mu^a(\Omegabar)$ 
we set $\mu:=\bar\mu^a=\bar\theta\,\HH^1\trace S$.

\smallskip
If $L>\bar\mu^a(\Omegabar)=\int_S \bar\theta\, d\HH^1$ we set 
$\mu:=\theta\,\HH^1\trace S$ where $\theta$ is a any function 
such that $\theta\ge\bar\theta$ and 
$L=\int_S \theta\, d\HH^1=\mu(\Omegabar)$.

\smallskip
Let us now prove the second part of the statement.
Since $\bar\mu^a$ is a solution of problem~\eqref{maxbl}, 
Lemma~\ref{sharp} implies that $\bar\mu^a(\Omegabar)=L$. 
On the other hand $\bar\mu(\Omegabar)\le L$ because of the 
definition of $\ML$, and therefore we must have $\bar\mu(\Omegabar)=L$
and $\bar\mu=\bar\mu^a$, which concludes the proof.
\end{proof}

\section{Some necessary conditions of optimality}
\label{snece}
In this section we assume that the load $f$ belongs to $L^2(\Omega)$, 
and we consider a measure $\mu=\theta\HH^1\trace S$ in $\MLa$ 
(see \eqref{mla}) and the function $u\in H^1_0(\Omega)\cap H^1(S)$ 
that solves problem~\eqref{emu2}, that is, the unique minimizer 
of~$E^*_f(\mu,\cdot)$. 

In Proposition~\ref{necess1} we derive some necessary 
conditions that $\mu$ and $u$ must satisfy if $\mu$ solves the 
maximum problem~\eqref{maxbl}.

In Proposition~\ref{necess0} we derive the Euler-Lagrange 
equations for $u$ in strong form (assuming some regularity
on $S$ and $u$). 

\begin{proposition}
\label{necess1}
Assume that $\mu$ solves of the optimization problem~\eqref{maxbl} 
and that the set $\smash{ S_+:=\{x\in S \colon \theta(x)>1\} }$
has positive length. Then there exists a constant $c\in\R$ such that
\begin{itemize}
\item[\rm(i)]
$|\nablatau u|=c$ a.e.~on $S_+$;
\item[\rm(ii)]
$|\nablatau u|\le c$ a.e.~on $S\setminus S_+$.
\end{itemize}
\end{proposition}

\begin{proof}
The proof is divided in several steps; 
the key inequality is~\eqref{etaineq2}, 
which is obtained from~\eqref{etaineq}.

We consider variations of $\mu$ of 
the form $\mu_\eps:=(\theta+\eps\eta)\,\HH^1\trace S$, 
with $\eps>0$ and $\eta\in L^\infty(S)$ 
(in particular we keep the set $S$ fixed).
In order that $\mu_\eps$ be admissible, that is, 
$\mu_\eps\in\ML$ for $0 \le \eps\le 1$, 
we assume that
\begin{equation}
\label{admmue}
\text{$\int_S\eta\,d\HH^1=0$ and 
$\eta\ge 1-\theta$~a.e.}
\end{equation}

\emph{Step~1. Let $u_\eps$ be the minimizer of $E^*_f(\mu_\eps,\cdot)$:
then for every $\eta$ that satisfies \eqref{admmue} there holds}
\begin{equation}
\label{etaineq}
\int_S |\nablatau u_\eps|^2 \, \eta \, d\HH^1 \le 0
\, .
\end{equation}
By the choice of $u$ and $u_\eps$ we have that 
$\EE_f(\mu) = E^*_f(\mu,u)$ and $\EE_f(\mu_\eps)=E^*_f(\mu_\eps,u_\eps)$.
Therefore, the optimality of $\mu$ yields
\begin{equation}
\label{nece1}
E^*_f(\mu,u_\eps)
\ge E^*_f(\mu,u)
\ge E^*_f(\mu_\eps,u_\eps) 
= E^*_f(\mu,u_\eps) 
+ \frac{m\eps}{2}\int_S |\nablatau u_\eps|^2 \, \eta \, d\HH^1
  \,,
\end{equation}
and the comparison of the first and last terms of \eqref{nece1}
gives~\eqref{etaineq}.

\medskip
In the next four steps we prove that the functions $u_\eps$ 
converge strongly to $u$, which will imply that
\eqref{etaineq} holds with $u$ in place of $u_\eps$.

\smallskip
\emph{Step~2. The functions $u_\eps$ are uniformly bounded 
in $H^1_0(\Omega)\cap H^1(S)$.}
Note indeed that for $\eps$ small enough there holds
$1/2 \le \theta+\eps\eta$ and therefore
\begin{align*}
  \frac{1}{2}\int_\Omega |\nabla u_\eps|^2\,dx
  + \frac{m}{4}\int_S |\nablatau u_\eps|^2 & \, d\HH^1 
  - \int_\Omega fu_\eps \, dx \\
& \le E^*_f(\mu_\eps, u_\eps) 
  \le E^*_f(\mu_\eps, 0) =0
\, ,
\end{align*}
and the functional in the first line
is clearly coercive on~$H^1_0(\Omega)\cap H^1(S)$. 

\smallskip
\emph{Step~3. $E^*_f(\mu, u_\eps)$ converge to $E^*_f(\mu,u)$ as $\eps\to 0$.}
From~\eqref{nece1} we obtain 
\[
0 
\le E^*_f(\mu,u_\eps) - E^*_f(\mu,u)
\le - \frac{m\eps}{2}\int_S |\nablatau u_\eps|^2 \, \eta \, d\HH^1
\, , 
\]
and the last term tends to $0$ 
as $\eps\to 0$ by Step~2.

\smallskip
\emph{Step~4. The functions $u_\eps$ converge to $u$ 
weakly in $H^1_0(\Omega)\cap H^1(S)$
as $\eps\to 0$.}
By Step~3 and the weak lower-semicontinuity of $\smash{E^*_f(\mu,\cdot)}$, 
every weak* limit of the sequence $u_\eps$ is a minimizer of 
$\smash{E^*_f(\mu,\cdot)}$ and therefore it must be $u$ because 
this functional is strictly convex.

\smallskip
\emph{Step~5. The functions $u_\eps$ converge to $u$ 
strongly in $H^1_0(\Omega)\cap H^1(S)$
as $\eps\to 0$, and for every $\eta$ that 
satisfies \eqref{admmue} there holds}
\begin{equation}
\label{etaineq2}
\int_S |\nablatau u|^2 \, \eta \, d\HH^1 \le 0
\, .
\end{equation}
Since the linear term in $E^*_f(\mu,\cdot)$ is weakly continuous, 
the convergence of the energies in Step~3 implies the convergence 
of the energies without linear term, that is
$E^*_0(\mu,u_\eps) \to E^*_0(\mu, u)$.
Notice now that
\[
\Phi(\cdot):=\big(E^*_0(\mu, \cdot)\big)^{1/2}
\]
is an equivalent Hilbert norm on $H^1_0(\Omega)\cap H^1(S)$, 
and recall that in Hilbert spaces weak convergence plus 
convergence of the norms implies strong convergence.

Inequality~\eqref{etaineq2} follows from~\eqref{etaineq}
and the fact that the functions $|\nablatau u_\eps|^2$ converge 
to $|\nablatau u|^2$ in $L^1(S)$.

\smallskip
\emph{Step~6. Conclusion of the proof.}
Set $S_\delta:=\{x\in S\colon \theta(x)\ge 1+\delta\}$.
Note that inequality \eqref{etaineq2} holds for all 
$\eta$ which vanish on $S\setminus S_\delta$, 
satisfy $|\eta|\le \delta$ on $S_\delta$, and have integral 
$0$ on $S_\delta$. Since this class of functions is closed
by change of sign, the inequality is actually 
an equality, which can be written as 
\[
\int_{S_\delta} |\nablatau u|^2 \, \eta \, d\HH^1 = 0
\,,
\]
and implies that $|\nablatau u|^2$ is equal to some constant $c$
a.e.~on $S_\delta$. Since this holds for every $\delta>0$, 
we have proved statement~(i).

Using statement~(i) and recalling that for every admissible 
$\eta$ there holds
\[
\int_{S\setminus S_+} \hskip-5pt \eta \, d\HH^1= -\int_{S_+} \eta \, d\HH^1
\, , 
\]
we rewrite \eqref{etaineq2} as
\[
0 \ge 
\int_{S\setminus S_+} \hskip-5pt |\nablatau u|^2 \, \eta \,d\HH^1 
    + c^2\int_{S_+} \hskip-5pt \eta\,d\HH^1
=\int_{S\setminus S_+} \hskip-5pt \big(|\nablatau u|^2-c^2\big)\, \eta \, d\HH^1
\,,
\]
and since the restriction of $\eta$ to $S\setminus S_+$ can be an arbitrary
positive bounded function with integral less than $\int_S (\theta-1) \, d\HH^1$, 
this inequality implies that $|\nablatau u|^2-c^2 \le 0$ a.e.~on 
$S\setminus S_+$, which is statement~(ii).
\end{proof}

For the next result we need some
assumptions on $S$, $\theta$ and $u$.

We assume that $\theta$ is a continuous function
and that $S$ is a network of class $C^1$, 
that is, it can be written as a finite union of simple 
curves $S_i$ of class $C^1$ contained in $\Omegabar$ that 
intersect each other and $\bd\Omega$ only at the endpoints. 
We denote by $S^\#$ the set of all endpoints of the curves
$S_i$, and we say that $x\in S^\#$ is 
\begin{itemize}
\item
a boundary point if $x\in\bd\Omega$;
\item
a terminal point if $x\in\Omega$ and 
$x$ belongs to only one curve $S_i$;
\item
a branching point if $x\in\Omega$ and 
$x$ belongs to more than one curve $S_i$.
\end{itemize}
We choose an orientation $\tau$ of $S$,%
\footnote{
This means that $\tau$ agrees on each curve $S_i$ (except 
the endpoints) with is a continuous unit tangent field 
to $S_i$; we do not require that $\tau$ is continuous 
at branching points.
}
we denote by $\nu$ the associated normal, that is, the 
rotation of $\tau$ by $90^\circ$~counterclockwise, 
and write $\bd_\tau$ for the tangential derivative, 
$\bd_\nu^\pm$ for the normal derivatives
on the two sides of~$S$. 

Finally we assume that $u$ is of class $C^1$ on 
$\Omega\setminus S$, and that the normal derivatives
$\bd_\nu^\pm u$ esist at every point of $S\setminus S^\#$
and belong to $L^1(S)$. 
We write
\[
\big[ \bd_\nu u\big] := \bd_\nu^+ u - \bd_\nu^+ u
\, .
\]
(Note that this quantity does not depend on the 
choice of the normal~$\nu$.)

\begin{proposition} 
\label{necess0}
Under the assuptions on $S$, $\theta$ and $u$ stated above, we have that 
\begin{itemize}
\item
$u$ solves $-\Delta u=f$ on $\Omega\setminus S$ with boundary condition 
$u=0$ on $\bd\Omega$;
\item
$u$ solves $-m \, \bd_\tau\big(\theta\,\bd_\tau u\big)
= \big[\bd_\nu u\big]$ on each curve $S_i$ minus the endpoints;
\item
$u$ is of class $C^1$ on each curve $S_i$, including the endpoints.
\end{itemize}
In particular the values of $\bd_\tau u$ at the endpoints of $S_i$, 
denoted by $(\bd_\tau u)_i$ are well-defined, and for every 
$x\in S^\#$ we set
\[
\big[\bd_\tau u(x)\big]
:=\sum (\bd_\tau u(x))_i
\, ,
\]
where the sum is taken over all $i$ such that $x$ is an endpoint of 
$S_i$. Then
\begin{itemize}
\item
if $x$ is a boundary point, the Dirichlet 
condition $u(x)=0$ holds;
\item
if $x$ is a terminal point, the 
Neumann condition $\bd_\tau u(x)=0$ holds;
\item
if $x$ is a branching point, 
the Kirchhoff condition
$\big[\bd_\tau u\big(x)]=0$
holds.
\end{itemize}
\end{proposition} 

\begin{proof}
The full Euler-Lagrange equation for $u$ 
in the weak form is
\begin{equation}
\label{weakform}
\int_\Omega\nabla u \cdot \nabla\phi\,dx 
 + m\int_S \nablatau u \cdot \nablatau\phi \, \theta \, d\HH^1 
 - \int_\Omega f\phi\,dx=0
\quad\forall\phi\in C^\infty_c(\Omega)
  \,.
\end{equation}
Thus $u$ satisfies the equation 
$\Delta u=f$ on $\Omega\setminus S$ in the weak sense, 
and hence it belongs to $\smash{ H^2_{\mathrm{loc}}(\Omega\setminus S) }$.

Integrating by parts the first integral 
in \eqref{weakform} we obtain
\[
\int_\Omega \nabla u \cdot \nabla\phi \,dx
=\int_{\Omega\setminus S}\nabla u \cdot \nabla\phi \,dx
=\int_{\Omega\setminus S} f\phi\,dx 
 - \int_S \big[\bd_\nu u\big] \phi\, d\HH^1
\, ,
\]
and therefore \eqref{weakform} becomes
\begin{equation}
\label{weakform2}
m\int_S \bd_\tau u \cdot \bd_\tau\phi \, \theta \, d\HH^1 
  - \int_S\big[\bd_\nu u ] \, \phi\, d\HH^1
=0
\quad\forall\phi\in C^\infty_c(\Omega)
\,.
\end{equation}
Thus $u$ solves the equation 
$-m \, \bd_\tau (\theta\,\bd_\tau u)=\big[\bd_\nu u\big]$
in the weak sense on each curve $S_i$, 
which implies that $\theta\,\bd_\tau u$ belongs $W^{1,1}(S_i)$, 
and then also to $C^0(S_i)$, which in turn 
implies that $u$ belongs to $C^1(S_i)$.

Finally we integrate by parts the first integral in \eqref{weakform2} 
and obtain
\[
\sum_{x\in S^\#} \theta(x) \, \big[\bd_\tau u(x)\big] \, \phi(x)
= 0 
\quad\forall\phi\in C^\infty_c(\Omega) 
\, ,
\]
This implies that $\big[\bd_\tau u(x)\big] = 0$ for every 
$x\in S^\#$ which is not a boundary point; 
if $x$ is a terminal point this means $\bd_\tau u(x)=0$ .
\end{proof}

\section{Numerical approximation of optimal reinforcing networks}
\label{snum}
In this section we introduce a numerical strategy to approximate the 
solutions of the relaxed reinforcement problem \eqref{maxbl}.
Through this section we assume that $\Omega$ is a bounded \emph{convex}
domain, and that the load $f$ belongs to $L^2(\Omega)$. 

Thanks to Theorem~\ref{exthsl}
we can rewrite this optimization problem as
\begin{equation}
\label{pbm}
\max_{S,\theta} \, \min_u 
\left[
\frac{1}{2}\int_\Omega|\nabla u|^ 2 \, dx
+ \frac{m}{2} \int_S |\nabla u|^ 2 \, \theta \, d\HH^1 
- \int_\Omega f u \, dx
\right]
\, , 
\end{equation}
where the minimum is taken over all function 
$u\in H^1_0(\Omega)\cap H^1(S)$
and the maximum is taken over all $S$ and $\theta$ such that
$S$ is a compact, connected set with finite length contained 
in $\Omegabar$, $\theta$ is a function on $S$ with
$\theta\ge 1$ a.e.\ and $\int_S \theta\, d\HH^1=L$.


Since we expect problem \eqref{pbm} to have many local maxima, 
we focus on stochastic optimization algorithms which only 
require cost function evaluations to proceed.

\subsection{Spanning tree parametrization and discrete functional}
To discretize problem \eqref{pbm}, we consider a mesh $\T$ 
associated to the domain $\Omega$ made of $n_p$ points and 
$n_t$ triangles. 
We denote by $K$ and $M$ respectively the stiffness and mass 
matrices of dimensions $n_p\times n_p$ associated to the 
finite elements $P1$ on~$\T$. 
Moreover, we define $K_x$ and $K_y$ to be the differentiation 
matrices of $P1$ functions. 
More precisely, $K_x$ and $K_y$ are matrices of dimensions 
$n_t \times n_p$ which evaluate the operators $\bd_x$ and 
$\bd_y$ on piecewise linear continuous functions on the mesh~$\T$. 
Observe that due to the linearity of $P1$ elements, $\bd_x$ and 
$\bd_y$ are constant on every triangle of the mesh. 

Denoting by $V_\areas$ the column vector of size $n_t\times1$ 
containing the area measures of every triangle, we recall the 
simple identity
\[
K = K_x^T V_\areas K_x + K_y^T V_\areas K_y
\, .
\]
We denote by the letter $U$ a real vector of $n_p$ node values 
representing an element of $P1\cap H^1_0(\Omega)$.

Problem~\eqref{pbm} involves both a connected set and an 
associated weight function. 
In order to parametrize connected one dimensional structures, 
we follow the strategy developed in~\cite{buttazzo2002optimal}. 
Take $n_d=1,2,\dots$ and consider a set of $n_d$ points 
$P_1, \dots, P_{n_d} \in\Omega$.
We associate to such a set its canonical spanning tree 
$SP(P_1, \dots, P_{n_d})$, 
which is the polygonal set of minimal length connecting 
these points without introducing new branching points. 
Let us point out that, generically, $SP(P_1,\dots,P_{n_d})$ 
is the union of $n_d-1$ arcs. 

It is straightforward to establish that the family of all 
such spanning trees (with $n_d$ varying among all integers) 
is dense with respect to Hausdorff distance 
among compact connected subsets of $\Omegabar$. 
To describe an $L^1$ element of $SP(P_1, \dots, P_{n_d})$, 
we simply consider a vector $\theta_\weights$ of $n_d - 1$ 
values greater than $1$ which represents a piecewise constant 
function on every arc of the tree.

Let $V_\lengths (P_1,\dots,P_{n_d},\theta_\weights)$ be the 
vector of size $n_t\times1$ which contains the weighted lengths 
of $SP(P_1,\dots,P_{n_d})$ intersected with every triangle 
of the mesh $\T$. 
With previous notations, we can now introduce a discrete 
approximation of problem~\eqref{pbm}:
\begin{align*}
\max \, \min \, \bigg[ \frac{1}{2} \, U^T K U
+\frac{m}{2} \, U^T \left( K_x^T V_\lengths K_x+K_y^T V_\lengths K_y\right) U-M F \bigg]
\end{align*}
where $F$ is the linear interpolation of the function 
$f$ at the vertices of the mesh~$\T$, 
the minimum is taken over all $U \in P1 \cap H^1_0(\Omega)$,
the maximum is taken over all pairs
$(SP(P_1,\dots,P_{n_d}),\theta_\weights)$,
that satisfy the constraints that every value of 
$\theta_\weights$ is greater than $1$ and the following 
measure equality holds:
\begin{equation}
\label{wconstr}
\sum_{i=1}^{n_d-1}\HH^1(S_i) \, \theta_\weights(i)= L
\, , 
\end{equation}
where the $(S_i)_{1\le i\le n_d}$ are the $n_d-1$ edges 
of $SP(P_1,\dots,P_{n_d})$. 

Since the minimization problem is a strictly 
convex quadratic problem, it reduces to solve the linear system
\begin{equation}
\label{pblinearsystem}
        \left[ K+\frac{m}{2} \left( K_x^T V_\lengths K_x +
        K_y^T V_\lengths K_y\right)\right] U = M F
\,.
\end{equation}

\subsection{Parametrization of the constraints}
\begin{algorithm}[t]
\caption{Projection on weighted length and bound constraints.}
\label{algo:proj}
\begin{description}
\item[Input] 
$L$, $SP(P_1,\dots,P_{n_d})$, $\theta_\weights$, $h_s$.
\item[step 1] 
Compute the length $\overline{L}$ of $SP(P_1,\dots,P_{n_d})$ and 
the center of mass $\overline{C}$ of the points $P_1,\dots,P_{n_d}$.
\item[step 2] 
Define $(\overline{P_1},\dots,\overline{P_{n_d}})$ to be the image 
of $SP(P_1,\dots,P_{n_d})$ by the homothetic transformation 
with center $\overline{C}$ and ratio ${h_s L}/{\overline{L}}$.
\item[step 3] 
Project the weight vector $\theta_\weights$ on the convex set which 
is the intersection of the linear constraint \eqref{wconstr} with 
respect to $SP(\overline{P_1},\dots,\overline{P_{n_d}})$ and the 
bound constraints $\theta_\weights \ge1$. 
The projected vector is denoted by $\overline{\theta_\weights}$.
\item[Output] 
$SP(\overline{P_1},\dots,\overline{P_{n_d}}),\,\overline{\theta_\weights}$.
\end{description}
\end{algorithm}

As explained in the previous sections, we need the couple 
$(SP(P_1,\dots,P_{n_d}),\theta_\weights)$ to have weights greater 
than one and satisfies equality constraint \eqref{pblinearsystem}. 
To parametrize such admissible couples we introduce a last scale 
parameter denoted by $h_s\in(0,1)$. 
We introduce in Algorithm~\ref{algo:proj} a three steps procedure 
to produce an admissible pair 
$\smash{\big(SP(\overline{P_1},\dots,\overline{P_{n_d}}\big),\overline{\theta_\weights})}$ 
for a given triplet of parameters 
$\smash{\big(SP(P_1,\dots,P_{n_d}),\theta_\weights,h_s\big)}$.

\subsection{Technical details and complexity}
We summarize in Algorithm~\ref{algo:cost} the different 
steps required to compute the cost associated to a given 
set of parameters, that we choose as 
$\big(SP(P_1,\dots,P_{n_d}),\theta_\weights,h_s\big)$. 

We give below some technical details and underline the 
computational complexity of every step.

In the first phase of projection, only the final step 
of the procedure is not computationally trivial. 
Whereas the projection of a point onto an hyperplane can 
be analytically described, the projection on an hyperplane 
intersected with a box requires a specific attention.
In all our experiments, we used Dai and Fletcher 
algorithm~\cite{dai2006new} to obtain a fast and 
precise approximation of this projection. 

Observe that the spanning tree 
$SP(\overline{P_1},\dots,\overline{P_{n_d}})$ is 
precisely by construction of length $h_s L\le L$ 
which implies that constraints \eqref{wconstr} 
and $\theta_\weights\ge1$ are compatible. 
In our situation, an order of only $n_d$ iterations 
was required to reach a relative error of $10^{-6}$ 
on first order optimality conditions with respect 
to the infinity norm which reduces to a complexity 
of order $n_d^2$.

The second and third steps have been carried out using an 
hash structure representation of the mesh $\T$ 
combined with a Quad-tree associated to its vertices. 
Using those precomputed information, these operations 
required in practice an order of $(n_d+n_p)\log(n_d+n_p)$ 
operations.

Finally, assembling and solving the linear system has 
been performed by a standard Cholesky decomposition 
which concentrated the main part of the computational 
effort in our experiments where the number of parameters 
$3n_d$ was negligible with respect to $n_p$ which was 
of order $10^4$.

\begin{algorithm}[t]
\caption{Summary of one cost evaluation.}
\label{algo:cost}
\begin{description}
\item[Input] 
$m,\, l,\,$ $SP(P_1, \dots, P_{n_d}),\, \theta_\weights,\, h_s$.
\item[step 1] 
Project $(SP(P_1, \dots, P_{n_d}),\, \theta_\weights)$ 
with Algorithm~\ref{algo:proj} to obtain an admissible couple
$(SP(\overline{P_1}, \dots, \overline{P_{n_d}}),\, \overline{\theta_\weights})$.
\item[step 2] 
Locate points $\overline{P_1}, \dots, \overline{P_{n_d}}$ 
in the mesh $\T$.
\item[step 3]
Compute the intersection of every arc of 
$SP(\overline{P_1}, \dots,\overline{P_{n_d}})$ 
with every triangle of $\T$ to evaluate
$V_\lengths(\overline{P_1},\dots,\overline{P_{n_d}},\,\overline{\theta_\weights})$.
\item[step 4] 
Assemble matrix $K_x^T V_\lengths K_x + K_y^T V_\lengths K_y$
and solve linear system \eqref{pblinearsystem} to 
compute its solution $\overline{U}$.
\item[Return] 
$\frac{1}{2} \,\overline{U}^T \! K\overline{U} 
+ \frac{m}{2} \,\overline{U}^T \!\!
  \left( K_x^T V_\lengths K_x+K_y^T V_\lengths K_y \right)\overline{U}-M F$
\end{description}
\end{algorithm}

\subsection{Numerical experiments}
Based on previous discretization, we approximate optimal 
triplet solutions $(S,\theta,u)$ 
of problem~\eqref{pbm} using a stochastic 
algorithm. 
We focus our study on the homogeneous load case 
corresponding to $f$ constantly equal to $1$ and on the sum 
of two Dirac masses $f=\delta_{(-1/2,0)}-\delta_{(1/2,0)}$. 

In all our experiments, we used the NLopt library 
(see~\cite{nlopt}) and its implementation of ISRES algorithm 
with its default parameters which combine local and global 
stochastic optimization. 

We carried out optimization runs limited to five hours 
of computation leading to an order of $2\times10^6$ cost 
function evaluations based on algorithm~\ref{algo:cost} 
on a standard computer for a mesh made of $10^4$ triangles. 

In Figures~\ref{fig:cracks1} and \ref{fig:cracks2}
we describe the optimal configurations we obtained for 
$L=1$ to $L=6$ with $n_d=20$. 
Observe that the resulting number of parameters in the 
triplet is exactly $3n_d$. 
Moreover, in order to obtain a fine and stable description
of optimal structures, we performed a local optimization
step of the obtained structure increasing the number of
points to $n_d=50$.
We used the NLopt implementation of the BOBYQA algorithm
for this final step which does not require gradient base
information.

Finally, we give in Table~\ref{table:results} several
numerical estimates obtained on a fine mesh with $10^5$
elements of our computed sets and also of natural networks
which could be guess to be optimal.
As illustrated by these numerical values, neither the 
radius (for $L=1$), a diameter (for $L=2$), a triple 
junction (for $L=3$) or a cross for ($L=4$) seem to 
be optimal.

We recover the fact, described in Proposition~\ref{necess1},
that, for optimal structures, the tangential gradient of $u$ 
is almost constant where $\theta>1$ whereas we 
can observe drastic changes of magnitude where $\theta=1$ 
(see Figures~\ref{fig:cracks1}, \ref{fig:cracks2} 
and \ref{fig:cracks3}).

\begin{table}[h]
\centering
\begin{tabular}{c|l|c}
Length constraint      & Theoretical guesses                  &  Computed optimal networks \\
\hline
1                      & -0.179471  \text{(radius)}           &  {-0.178873}  \\
2                      & -0.165095  \text{(diameter)}         &  {-0.161944}  \\
3                      & -0.152676  \text{(star)}             &  {-0.149601}  \\
4                      & -0.141969  \text{(cross)}            &  {-0.138076}  \\
5                      &  \multicolumn{1}{c|}{-}              &  {-0.127661}  \\
6                      &  \multicolumn{1}{c|}{-}              &  {-0.117140}
\end{tabular}
\caption{Reinforcement values computed on a fine mesh 
of $10^6$ elements for classical and computed connected
sets for $m=0.5$}
\label{table:results}
\end{table}

\section{Remarks and open questions}
\label{sopen}
There are several remarks and open problems related to the optimization 
problem~\eqref{optpb} and the relaxed optimization problem~\eqref{maxbl}; 
we list below those we deem more interesting.

\begin{remark}
\label{efdisc}
In general the functional $\EE_f(\cdot)$ 
is not weakly* continuous on $\ML$.
We prove this claim by an explicitly example.
We let $S\subset\Omega$ be a closed segment with 
length $2\delta$, which we identify with the interval $[-\delta,\delta]$,
and let $f$ be a signed measure of the form 
$f:=\rho\,\HH^1\trace S$ where $\rho$ is a 
function on $S$ with integral $0$. 

We then consider the measures $\mu_n:=\theta_n\,\HH^1\trace S$ 
where $\theta_n(s):=g(ns/\delta)$ and $g$ is the $2$-periodic function 
on $\R$ defined by $g=1$ on $[-1,0)$ and $g=2$ on $[0,1)$. 
Thus $\mu_n$ converge to $\mu:=\frac{3}{2}\,\HH^1\trace S$. 
However, the functionals 
\[
F(\mu_n,u)
:=\frac{1}{2}\int_S |\nablatau u|^2 d\mu_n - \int_S u\, df
=\int_{-\delta}^\delta \frac{\theta_n}{2} |\dot u|^2-\rho u\,ds
\]
Gamma-converge (on $H^1(S)$ endowed with the weak topology) to 
\[
F(u):=\int_{-\delta}^\delta \frac{2}{3} |\dot u|^2 - \rho u \, ds
\, , 
\]
and $F(u)< F(\mu,u)$ for every non constant $u$
(because $2/3$ is strictly less than $3/4$, which is the 
density of $\mu$ divided by $2$).
In particular if $f$ is not a.e.~equal to $0$ then 
\[
\lim_{n\to\infty} \min_u F(\mu_n,u)
= \min_u F(u)
< \min_u F(\mu,u)
\]
(all minima are taken over $u\in H^1(S)$).
Using the strict inequality we can prove that if the constant 
$m$ that appears in \eqref{emuu} is sufficiently large, then
\[
\limsup_{n\to\infty} \EE_{mf}(\mu_n) < \EE_{mf}(\mu)
\, .
\]
\end{remark}

\begin{problem}
\label{relax}
We do not know if problem~\eqref{maxbl} is 
the relaxation of problem~\eqref{optpb}. 
In other words, we do not know if the following 
approximation property holds:
\emph{for every $\mu\in\ML$ there exists a sequence of sets 
$S_n\in\AL$ such that
\begin{equation}
\label{enapprox}
\HH^1\trace S_n \to \mu
\quad\text{and}\quad
\EE_f(S_n) \to \EE_f(\mu)
\, .
\end{equation}
}%
Indeed, by the definition of $\ML$ every $\mu$ 
in this class is the limit of $\HH^1\trace S_n$
for some sequence of sets $S_n\in\AL$, 
but since $\EE_f$ is not continuous (Remark~\ref{efdisc}), 
the second limit in \eqref{enapprox} does not necessarily hold.
\end{problem}

\begin{remark}
\label{lavrentiev}
If the approximation in energy \eqref{enapprox} does not hold, 
then some kind of Lavrentiev phenomenon may occur. 
This means that 
\begin{itemize}
\item
the value of the maximum/supremum
in the original optimization problem~\eqref{optpb} could be strictly 
smaller than the value of the maximum in the relaxed optimization 
problem~\eqref{maxbl};
\item
given a maximizing sequence $(S_n)$ for problem~\eqref{optpb}, the 
associated measures $\HH^1\trace S_n$ may not converge to a 
solution of the relaxed problem~\eqref{maxbl}.
\end{itemize}
\end{remark}

\begin{remark}
\label{bobuserel}
Assume that $f$ belongs to $L^p(\Omega)$ for some $p>1$ and 
that $\mu$ is a measure in $\ML$ with support $S$, and let 
$\mu^a$ be the absolutely continuous part of $\mu$
with respect to~$\HH^1\trace S$.
Using Lemma~\ref{killsinglemma2}, Lemma~\ref{killnorlemma} and
Proposition~\ref{emumin} we easily obtain the following: 
\emph{the relaxation of $E_f(\mu,u)$ 
with $u\in C^\infty_c(\Omega)$ is
the functional $\smash{ E_f^*(\mu^a,u) }$ with
$\smash{ u\in H^1_0(\Omega)\cap H^1(S) }$.}

Notice that for $f=0$ we can rewrite $E_0(\mu,u)$ as 
\[
F(u):=\frac{1}{2} \int |\nabla u|^2 d\lambda
\] 
where $\lambda:=dx+ m\mu$, $dx$ is the Lebesgue measure on $\Omega$, 
and $m$ is the number that appears in \eqref{emuu}. 
Functionals of this type has been studied in detail
in \cite{bobuse}, where it is proved that the relaxation of 
$F(u)$ with $u\in C^\infty_c(\Omega)$ is 
\[
F^*(u):=\frac{1}{2} \int |\nabla\!_\lambda u|^2 d\lambda
\, , \quad 
u\in H^1(\lambda)
\, ,
\]
where the space $H^1_\lambda$ 
and the operator $\nabla\!_\lambda$ are defined in a suitable
abstract sense.

Thus the relaxation result stated above can be rephrased as follows:
the space $H^1_\lambda$ agrees with $H^1_0(\Omega)\cap H^1(S)$
and the operator $\nabla\!_\lambda$ agrees with
the full gradient $\nabla$ for Lebesgue-a.e.~$x$, 
with the tangential gradient $\nablatau$ for 
$\HH^1$-a.e.~$x\in S$, and with the null-operator
for $\mu^s$-a.e.~$x$, where $\mu^s$ is
the singular part of $\mu$ w.r.t.~$\HH^1\trace S$.
\end{remark}

\begin{problem}
We denote by $\mu=\theta\HH^1\trace S$ a solution 
of problem~\eqref{maxbl} given in Theorem~\ref{exthsl}, 
and by $u$ the unique minimizer of $E_f^*(\mu,\cdot)$.
Here are some open questions concerning $\mu$ and $u$.

\smallskip
(a)~Intuition tells that it is never convenient 
to use part of $S$ to reinforce 
the boundary the membrane, because it 
is already reinforced by the Dirichlet boundary
condition inscribed in the problem. 
On the other hand, the requirement that $S$ be connected 
might force part of it to lie on the boundary of $\Omega$,
even if this part does not contribute to reinforcing
the membrane.
Here are two plausible statements that would be
interesting to investigate: 
\begin{itemize}
\item
for some non-convex domain $\Omega$ the set $S\cap\bd\Omega$ may have 
positive length, but $S$ cannot be entirely contained 
in $\bd\Omega$;
\item
if $\Omega$ is strictly convex then the set $S\cap\bd\Omega$ 
has zero-length, and perhaps it is even finite. 
\end{itemize}
Note that using the second part of Theorem~\ref{exthsl} 
(and in particular assuming that the support of $f$ is 
$\Omega$)
we can prove the following: if $\Omega$ is strictly convex then 
$S\cap\bd\Omega$ does not contain any arc.

\smallskip
(b)~In principle the density $\theta$ belongs to $L^1(S)$. 
It would be interesting to investigate if 
$\theta$ is bounded and, possibly refining the 
assumptions on the data, prove further regularity properties.

\smallskip
(c)~According to the numerical simulations we made, the set 
$S$ never contains closed curves; it would be interesting 
to show this fact under general assumptions.

\smallskip
(d)~Numerical simulations also show that $S$ may present 
branching points at least for values of $L$ large enough. 
However, the regularity of the set $S$ seems a difficult issue: 
is it true that, under suitable assumptions on the data, 
the set $S$ is smooth except a finite number of branching 
points?
And if a branching occurs, what are the necessary condition 
of optimality for the related angles?

\smallskip
(e)~When the support of $f$ is $\Omega$ and the total length 
$L$ tends to $+\infty$, then the optimal set $S$ tends to fill 
the entire $\Omega$. 
Can we say more on the asymptotic behavior of $S$
in this regime? 
This question is reminiscent of a $\Gamma$-convergence result 
for the irrigation problem proved in~\cite{mostil}.
\end{problem}

%
%

%
%
%
%
\vskip .5 cm

{\parindent = 0 pt\footnotesize
G.A.
\\
Dipartimento di Matematica, 
Universit\`a di Pisa,
largo Pontecorvo~5, 
56127 Pisa, 
Italy 
\\
e-mail: \texttt{giovanni.alberti@unipi.it}

\bigskip
G.B.
\\
Dipartimento di Matematica, 
Universit\`a di Pisa,
largo Pontecorvo~5, 
56127 Pisa, 
Italy 
\\
e-mail: \texttt{giuseppe.buttazzo@unipi.it}

\bigskip
S.G.L.B.
\\
Dipartimento di Matematica e Applicazioni, 
Universit\`a di Napoli ``Federico II'', 
via Cintia, Monte S.~Angelo, 
80126 Napoli, 
Italy
\\
e-mail: \texttt{serena.guarinolobianco@unina.it}

\bigskip
\'E.O.
\\
Laboratoire Jean Kuntzmann, 
Universit\'e Grenoble Alpes, 
38041 Grenoble, 
France
\\
e-mail: \texttt{edouard.oudet@imag.fr}

}

%
%
%
%
\newpage
\begin{figure}[ht]
\centering
\begin{tabular}{c}
      \includegraphics[height=6cm]{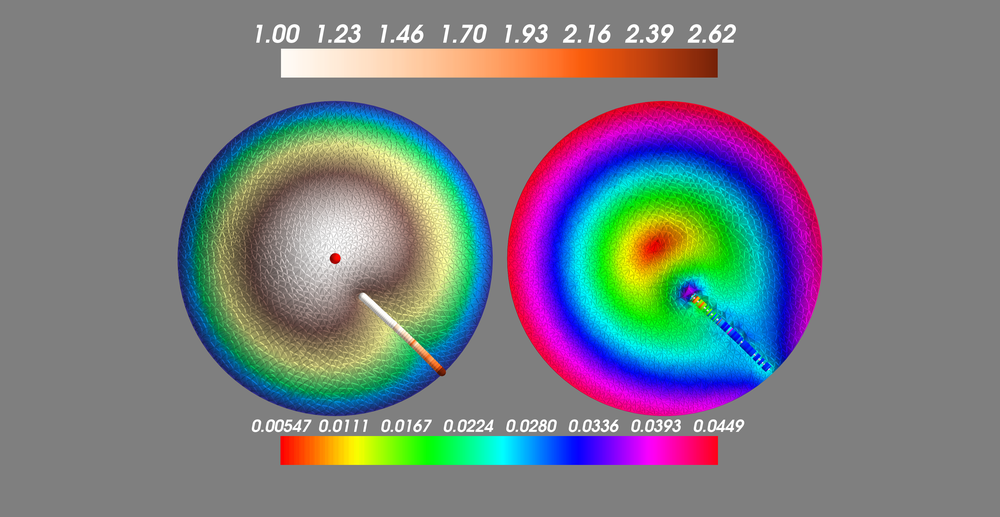}\\ 
      \includegraphics[height=6cm]{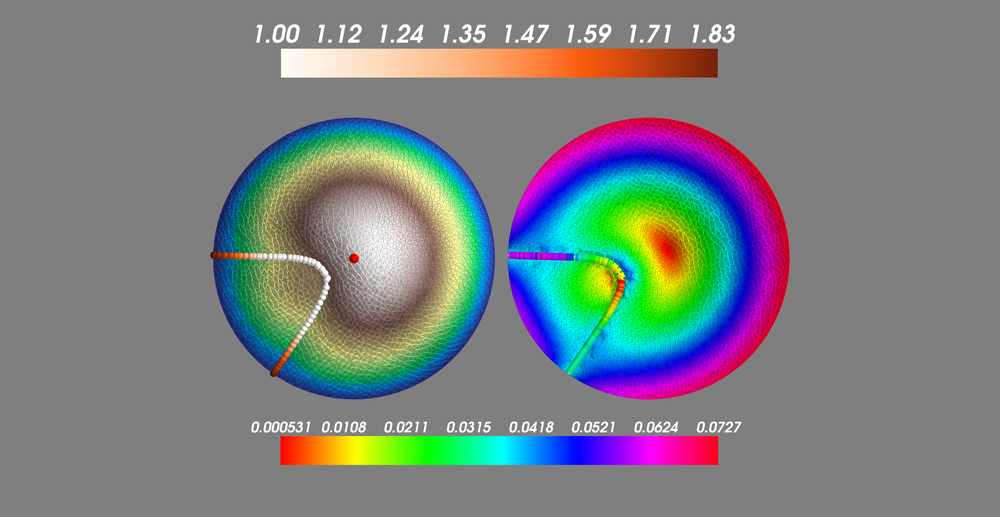}\\ 
      \includegraphics[height=6cm]{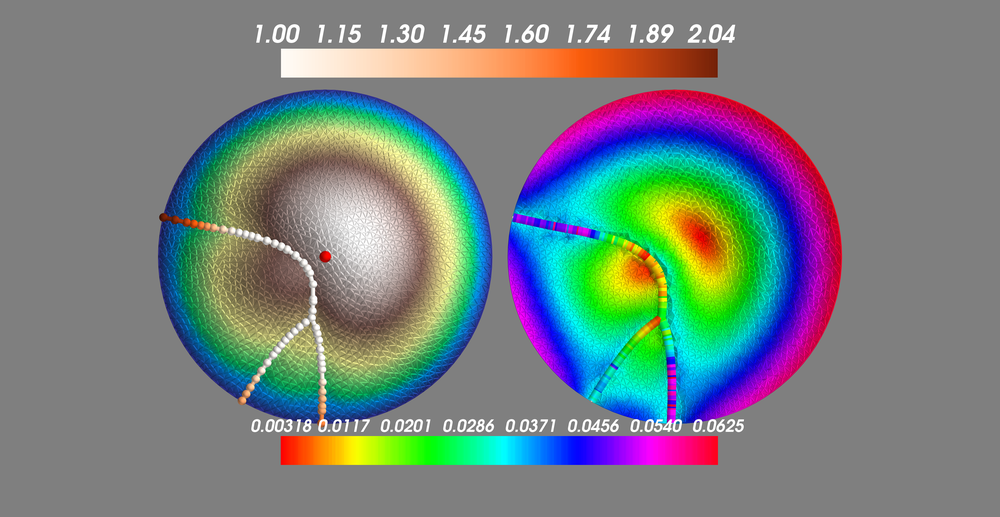}   
  \end{tabular}
  \caption{Approximation of globally optimal reinforcement
           structures for $m = 0.5$, $L = 1,\,2 $ and $3$. 
           The upper colorbar is related to the weights $\theta$ 
           which colors the optimal reinforcement set on the left, 
           whereas the lower colorbar stands for the tangential 
           gradient plotted on the connected set on the right picture} 
  \label{fig:cracks1}
\end{figure}

\newpage
\begin{figure}[ht]
  \centering
  \begin{tabular}{c}
     \includegraphics[height=6cm]{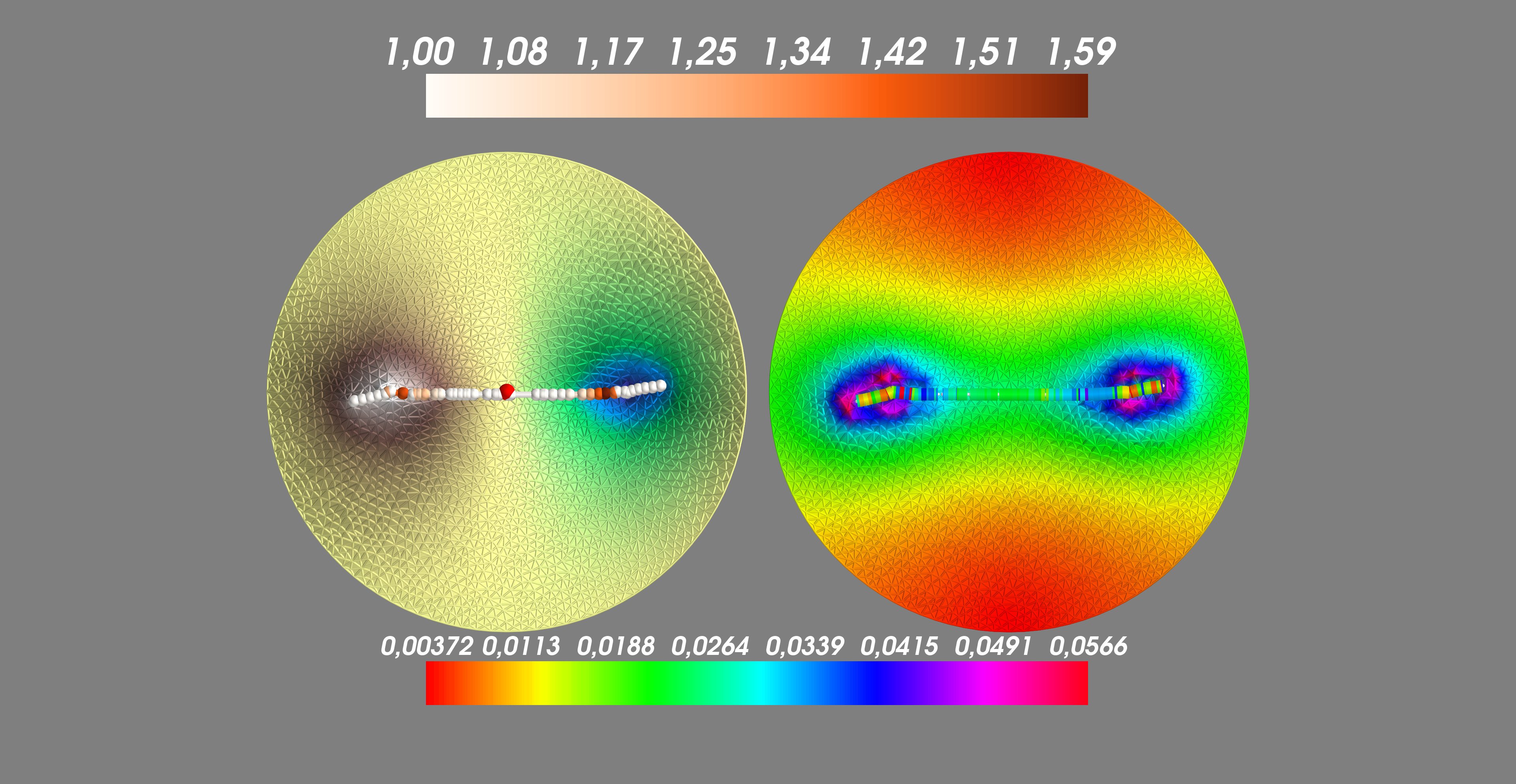} \\ 
     \includegraphics[height=6cm]{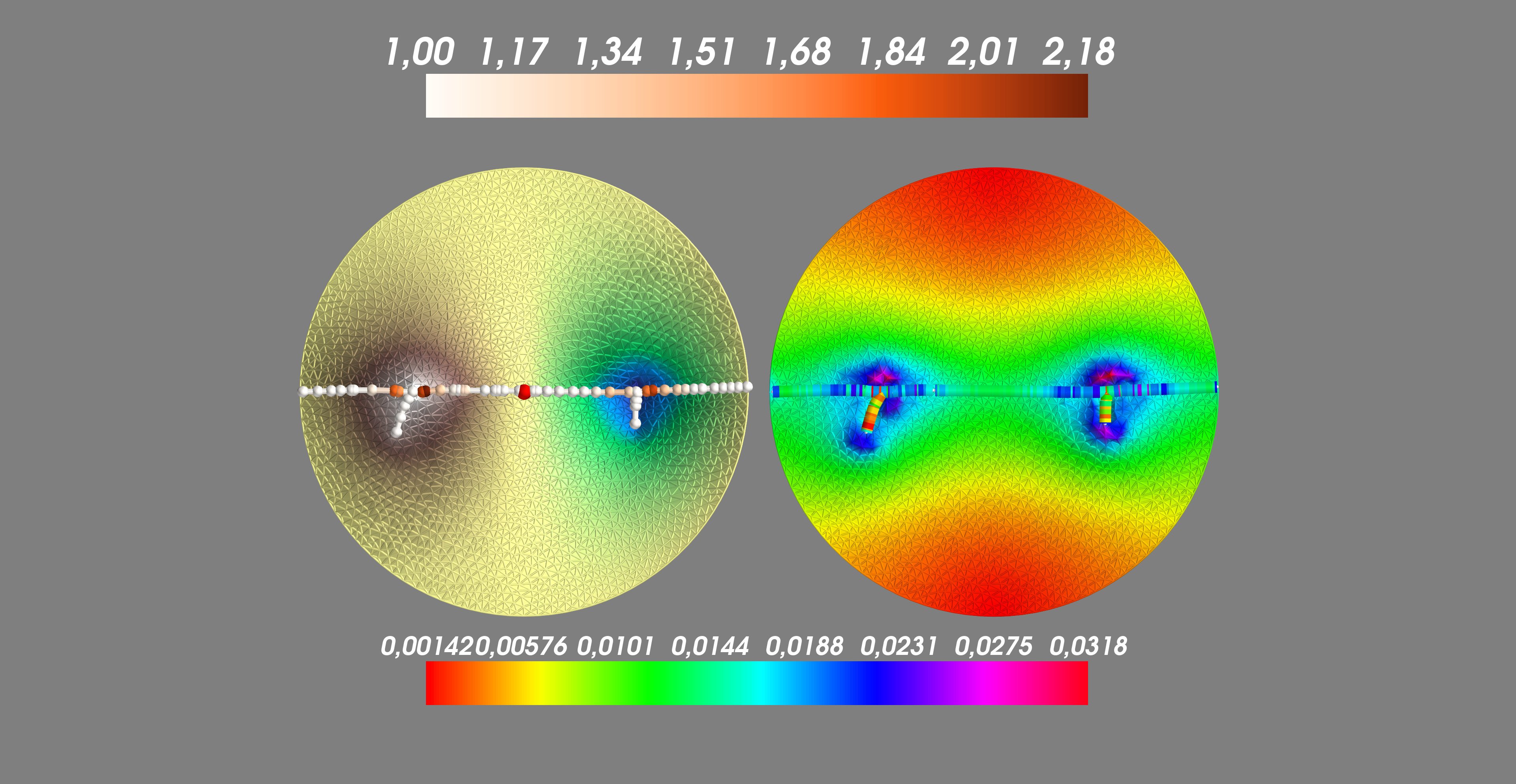} \\ 
     \includegraphics[height=6cm]{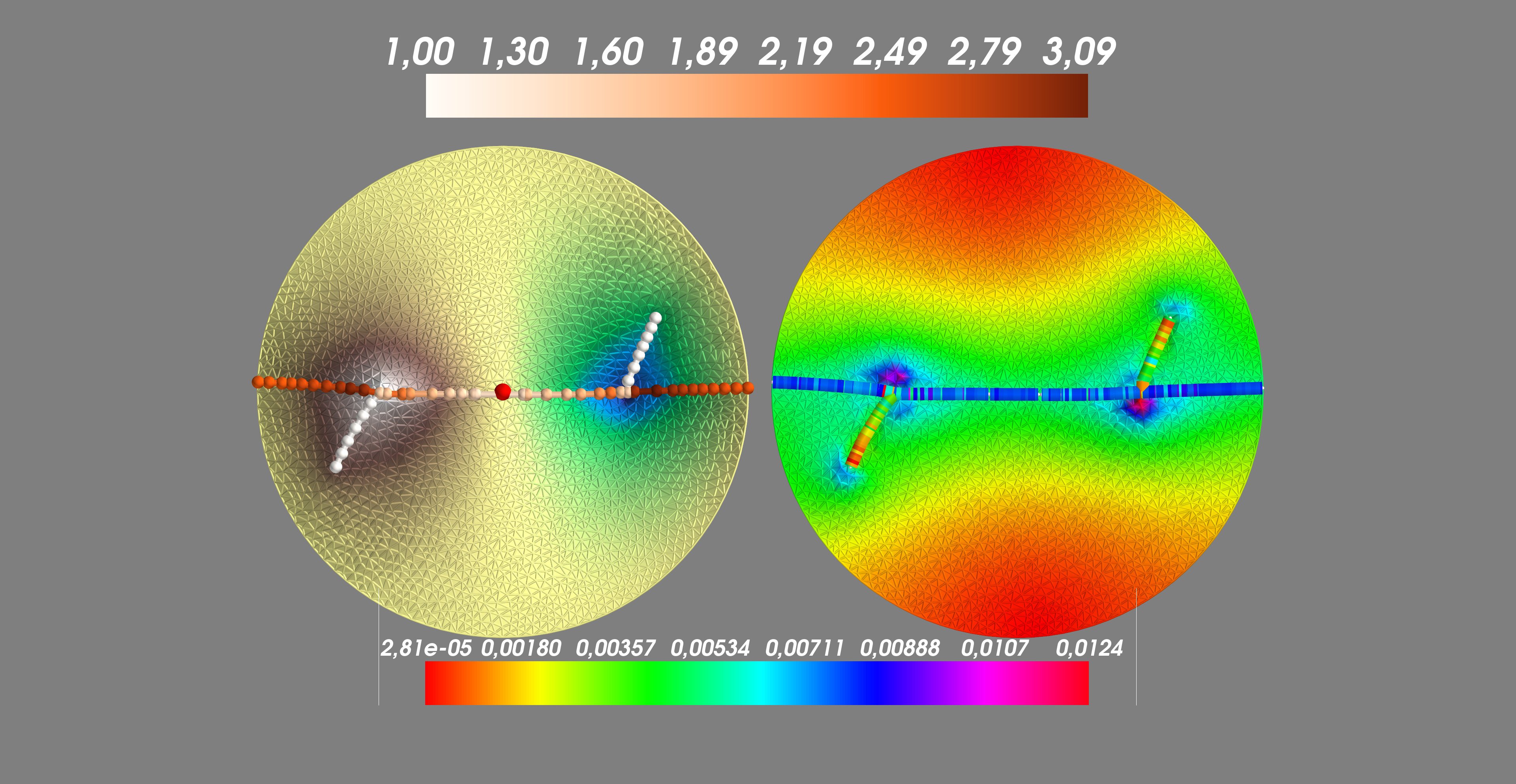}        
  \end{tabular}
  \caption{Approximation of globally optimal reinforcement
           structures for $m=0.5$, $L=1.5,\,2.5$ and $5$ for a source 
           consisting of two dirac masses. 
           The upper colorbar is related to the weights $\theta$ which 
           colors the optimal reinforcement set on the left, whereas the 
           lower colorbar stands for the tangential gradient plotted on 
           the connected set on the right picture}
  \label{fig:cracks3}
\end{figure}

\newpage
\begin{figure}[ht]
  \centering
  \begin{tabular}{c}
     \includegraphics[height=6cm]{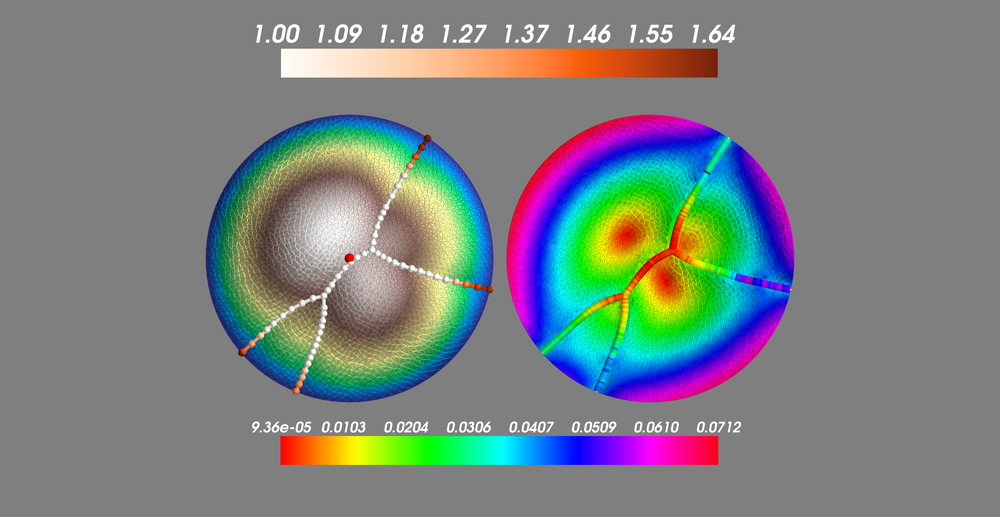}\\ 
     \includegraphics[height=6cm]{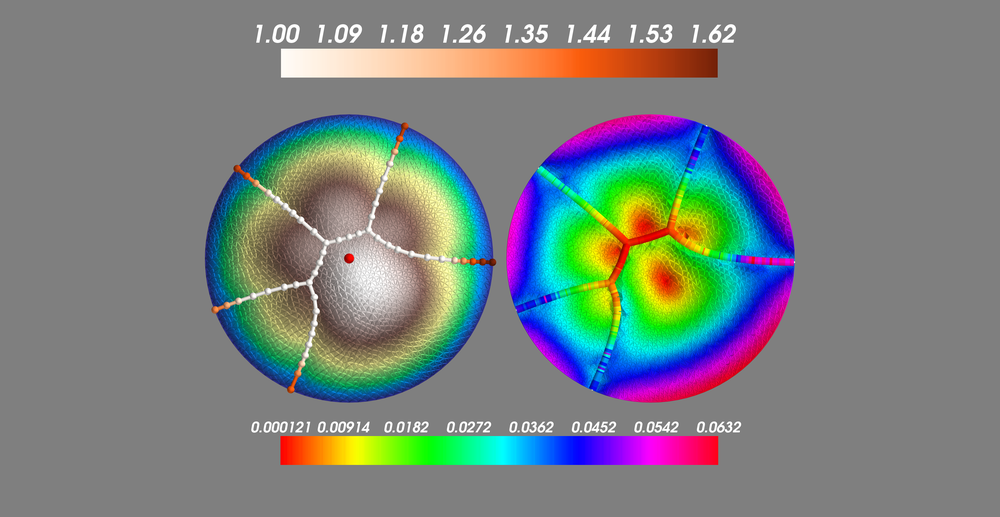}\\ 
     \includegraphics[height=6cm]{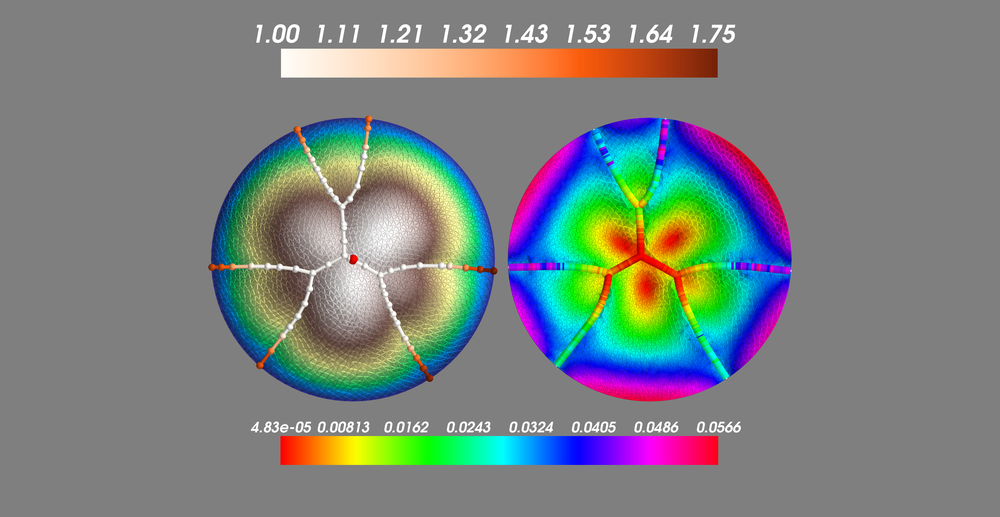}   
  \end{tabular}
  \caption{Approximation of globally optimal reinforcement
           structures for $m=0.5$, $L=4,\,5$ and $6$. 
           The upper colorbar is related to the weights $\theta$ 
           which colors the optimal reinforcement set on the left, 
           whereas the lower colorbar stands for the tangential gradient 
           plotted on the connected set on the right picture}
  \label{fig:cracks2}
\end{figure}


\begin{thebibliography}{99}
%
%
\newcommand{\bib}[3]{\bibitem{#1}{\sc#2:~}{#3.}}

\bib{albott}
{G.~Alberti, M.~Ottolini}
{On the structure of continua with finite length and \Golab's 
semicontinuity theorem.
{\it Nonlinear Anal.}, 153 (2017), 35--55}

\bib{acbupe88}
{E.~Acerbi, G.~Buttazzo, D.~Percivale}
{Thin inclusions in linear elasticity: a variational approach.
{\it J. Reine Angew. Math.}, 386 (1988), 99--115}


\bib{be52}
{M.~Beckmann}
{A continuous model of transportation.
{\it Econometrica}, 20 (1952), 643--660}

\bib{bobuse}
{G.~Bouchitt\'e, G.~Buttazzo, P.~Seppecher}
{Energies with respect to a measure and applications to low dimensional structures.
{\it Calc. Var. Partial Differential Equations}, 5 (1996), no.~1, 37--54}

\bib{bracarsan}
{L.~Brasco, G.~Carlier, F.~Santambrogio}
{Congested traffic dynamics, weak flows and very degenerate elliptic equations.
{\it J. Math. Pures Appl.}, 93 (2010), no.~6, 652--671}

\bib{bucagu}
{G.~Buttazzo, G.~Carlier, S.~Guarino Lo Bianco}
{Optimal regions for congested transport.
{\it ESAIM Math. Model. Numer. Anal.}, 49 (2015), no.~6, 1607--1619}

\bib{buouve15}
{G.~Buttazzo, \'E.~Oudet, B.~Velichkov}
{A free boundary problem arising in PDE optimization.
{\it Calc. Var. Partial Differential Equations}, 54 (2015), no.~4, 3829--3856}

\bib{buttazzo2002optimal}
{G.~Buttazzo, \'E.~Oudet, E.~Stepanov}
{Optimal transportation problems with free Dirichlet regions.
{\it Variational methods for discontinuous structures}, pp.~41--65. 
Progr. Nonlinear Differential Equations Appl.,~51. 
Birkh\"auser, Basel, 2002}

\bib{busa07}
{G.~Buttazzo, F.~Santambrogio}
{Asymptotical compliance optimization for connected networks.
{\it Netw. Heterog. Media}, 2 (2007), no.~4, 761--777}

\bib{busava06}
{G.~Buttazzo, F.~Santambrogio, N.~Varchon}
{Asymptotics of an optimal compliance-location problem.
{\it ESAIM Control Optim. Calc. Var.}, 12 (2006), no.~4, 752--769}

\bib{buva05}
{G.~Buttazzo, N.~Varchon}
{On the optimal reinforcement of an elastic membrane.
{\it Riv. Mat. Univ. Parma (Ser.~7)}, 4* (2005), 115--125}

\bib{dai2006new}
{Yu-Hong Dai, R.~Fletcher}
{New algorithms for singly linearly constrained quadratic programs 
subject to lower and upper bounds.
{\it Math. Program. (Ser.~A)}, 106 (2006), no.~3, 403--421}

\bib{falconer}
{K.J.~Falconer}
{{\it The geometry of fractal sets.}
Cambridge Tracts in Mathematics,~85.
Cambridge University Press, Cambridge, 1986}

\bib{golab}
{S.~\Golab}
{Sur quelques points de la th\'eorie de la longueur.
{\it Ann. Soc. Polon. Math.}, 7 (1929), 227--241}

\bib{nlopt}
{S.G.~Johnson}
{{\it The NLopt nonlinear-optimization package}.
Available at the webpage: {\tt http://ab-initio.mit.edu/nlopt}}

\bib{mostil}
{S.J.N.~Mosconi, P.~Tilli}
{$\Gamma$-convergence for the irrigation problem.
{\it J. Convex Anal.}, 12 (2005), no.~1, 145--158}

\bib{sapa80}
{E.~S\'anchez-Palencia}
{{\it Nonhomogeneous media and vibration theory}.
Lecture Notes in Physics,~127. 
Springer-Verlag, Berlin-New York, 1980}

\bib{wa52}
{J.G.~Wardrop}
{Some theoretical aspects of road traffic research.
{\it Proceedings of the Institution of Civil Engineers}, 
1 (1952), no~3, 325--362}

\end{thebibliography}
\end{document}